\documentclass[12pt]{amsart}

\usepackage{amsmath,amsthm,amsfonts,amssymb,eucal}

\usepackage{color}

\renewcommand {\a}{ \alpha }
\renewcommand{\b}{\beta}

\newcommand{\g}{\gamma}
\newcommand{\G}{\Gamma}

\renewcommand{\d}{\delta}
\newcommand{\s}{\sigma}
\renewcommand{\l}{\lambda}
\renewcommand{\L}{\Lambda}
\newcommand{\z}{\zeta}

\newcommand{\p}{\partial}

\newcommand{\Om}{\Omega}

\newcommand{\oq}{\ {\raise 7pt\hbox{${\scriptstyle\circ}$}}
	\kern -7pt{
		\hbox{$Q$}}}

\newcommand{\R}{ \mathbb R}

\newcommand {\bb}{\mathbf b}

\newcommand {\bx}{\mathbf x}

\newcommand {\be}{\mathbf e}

\newcommand {\bz}{\mathbf z}
\newcommand {\by}{\mathbf y}

\newcommand {\bn}{\mathbf n}

\newcommand {\boldeta}{\boldsymbol\eta}

\newcommand {\bxi}{\boldsymbol\xi}

\newcommand{\lu}{\langle}
\newcommand{\ru}{\rangle}

\newcommand{\CK}{\mathcal K}
\newcommand{\CB}{\mathcal B}

\newcommand{\CA}{\mathcal A}

\newcommand{\plainC}[1]{\textup{{\textsf{C}}}^{#1}}

\newcommand{\plainS}{\textup{{\textsf{S}}}}

\newcommand{\plainL}[1]{\textup{{\textsf{L}}}^{#1}}

\DeclareMathOperator{\tr}{{tr}}

\newcommand{\1}
{{\,\vrule depth3pt height9pt}{\vrule depth3pt height9pt}
	{\vrule depth3pt height9pt}{\vrule depth3pt height9pt}\,}

\DeclareMathOperator {\dist} {{dist}}

\DeclareMathOperator \meas {{meas}}

\DeclareMathOperator{\op}{{Op}}

\DeclareMathOperator{\supp}{{supp}}

\newtheorem{thm}{Theorem}[section]
\newtheorem{cor}[thm]{Corollary}
\newtheorem{lem}[thm]{Lemma}
\newtheorem{prop}[thm]{Proposition}
\newtheorem{cond}[thm]{Condition}

\theoremstyle{definition}

\newtheorem{rem}[thm]{Remark}

\numberwithin{equation}{section}

\newcommand{\bee}{\begin{equation}}
	\newcommand{\ene}{\end{equation}}
\newcommand{\bees}{\begin{equation*}}
	\newcommand{\enes}{\end{equation*}}
\newcommand{\bes}{\begin{split}}
	\newcommand{\ens}{\end{split}}

\newcommand{\bet}{\begin{thm}}
	\newcommand{\ent}{\end{thm}}
\newcommand{\bel}{\begin{lem}}
	\newcommand{\enl}{\end{lem}}
\newcommand{\bec}{\begin{cor}}
	\newcommand{\enc}{\end{cor}}
\newcommand{\bep}{\begin{proof}}
	\newcommand{\enp}{\end{proof}}
\newcommand{\ber}{\begin{rem}}
	\newcommand{\enr}{\end{rem}}

\begin{document}
	\hoffset -4pc

\title
[ The Szeg\H o formulas]
{{On the Szeg\H o formulas for truncated Wiener-Hopf operators}}
\author{Alexander V. Sobolev}
\address{Department of Mathematics\\ University College London\\
	Gower Street\\ London\\ WC1E 6BT UK}
\email{a.sobolev@ucl.ac.uk}

\begin{abstract}	 	
We consider functions of  multi-dimensional versions of truncated 
Wiener--Hopf operators with smooth symbols, and study 
the scaling asymptotics of their traces. The obtained 
results extend the asymptotic formulas obtained by H. Widom  in the 1980's 
to non-smooth functions, and non-smooth truncation domains. 

The obtained asymptotic formulas are used to analyse the 
scaling limit of the spatially 
bipartite entanglement entropy of thermal equilibrium states of non-interacting 
fermions at positive temperature.  
\end{abstract}

\maketitle

\section{Introduction}   

By the truncated Wiener-Hopf operator we understand the operator 
\begin{equation*}
W_\a  = W_\a(a; \L) = \chi_\L \op_\a(a) \chi_\L,\ \a>0,
\end{equation*}
where $\chi_\L$ is the indicator function of a 
region $\L\subset\R^d, d\ge 1$, and 
the notation $\op_\a(a)$ stands for the $\a$-pseudo-differential 
operator with the symbol $a = a(\bxi)$, i.e.
\begin{equation*}
\bigl(\op_\a(a) u\bigr)(\bx)
= \frac{\a^{d}}{(2\pi)^{\frac{d}{2}}}
\iint e^{i\a\bxi\cdot(\bx-\by)} a(\bxi) 
u(\by) d\by d\bxi,  u\in \plainS(\R^d).
\end{equation*} 
If the symbol $a$ is bounded then 
the operator $\op_\a(a)$, and hence $W_\a(a; \L)$, are bounded in $\plainL2(\R^d)$. 
Given a test function $f: \R\to\mathbb C$, we are interested in the
difference operator 
\begin{equation}\label{Dalpha:eq}
D_\a(a, \L; f) := \chi_\L f(W_\a(a; \L)) \chi_\L 
- W_\a(f\circ a; \L).
\end{equation}
Under appropriate conditions on $f, a$ and $\L$ this operator is trace class, 
and the subject of this paper is to study 
the trace of \eqref{Dalpha:eq} as $\a\to\infty$. 
We interpret the trace formulas to be obtained as 
``Szeg\H o asymptotic formulas" or ``Szeg\H o formulas", 
following the tradition that is traced back to 
the original G.Szeg\H o's papers 
\cite{Szego1915} and \cite{Szego1952}, see e.g. 
\cite{Widom1985} and references therein. 
The reciprocal parameter $\a^{-1}$ can be 
naturally viewed as Planck's constant, 
and hence the limit $\a\to\infty$ can be regarded as 
the quasi-classical limit. 
By a straightforward change of variables the operator \eqref{Dalpha:eq} 
is unitarily equivalent to $D_1(a, \a\L; f)$, so that the asymptotics  
$\a\to\infty$ can be also interpreted as a large-scale limit, 
which makes the term ``Szeg\H o asymptotics" even more 
natural. 

At this point we need to make one preliminary 
remark about the operator \eqref{Dalpha:eq} being trace class. If 
\begin{enumerate}
\item 
$\L$ is bounded, 
\item 
 the function $f$ is smooth and satisfies 
$f(0) = 0$, and
\item 
 the symbol $a$ decays sufficiently fast at infinity, 
\end{enumerate}
\noindent
then both operators on the right-hand side of \eqref{Dalpha:eq} can be easily shown to be trace class. 
However, as we see later, the difference \eqref{Dalpha:eq} may be trace class even without the conditions (1) and (2). 
%
In particular, 
being able to study unbounded $\L$'s 
is important for applications. 

The Szeg\H o type asymptotics for the truncated Wiener-Hopf operators 
for smooth bounded domains $\L$ and smooth functions $f$ 
have been intensively studied in the 1980's and early 1990's, 
see \cite{Widom1980}, \cite{Widom1982}, \cite{Roccaforte1984}, \cite{BuBu} 
and 
\cite{Widom1985} for further references. 
In particular, a full asymptotic expansion of 
$\tr D_\a(a, \L; f)$ in powers of 
$\a^{-1}$ was derived independently 
in \cite{BuBu} and in \cite{Widom1985}. 
We are concerned only with the leading term asymptotics: 
they have  the form 
\begin{align}\label{bd:eq}
\tr D_\a(a, \L; f) = \a^{d-1}(\CB_d(a)+o(1)), \a\to\infty,
\end{align}
where the coefficient $\CB_d(a)  = \CB_d(a; \p\L, f)$ is defined 
in \eqref{cbd:eq}. 
Our objective is to generalize this formula in two ways: namely, 
we extend it \\
-- to non-smooth functions $f$, such as, for example, $f(t) = |t|^\g$ 
with some $\g >0$, and\\ 
-- to piece-wise smooth regions $\L$. \\
The extension to non-smooth functions for 
$d = 1$ was implemented in \cite{Leschke2016}. 
In this paper we concentrate on the 
multi-dimensional case, i.e. on $d\ge 2$. 
The precise statement is contained in Theorem \ref{main:thm}. 

We need to emphasize a few points: 
\begin{enumerate}
\item
In the main theorem the non-smoothness conditions 
do not concern the symbol $a$: 
it 
is always assumed to be a $\plainC\infty$-function. 
\item 
In contrast to the results of \cite{BuBu} and \cite{Widom1985}, 
for non-smooth functions $f$ we are only able to establish the first term of the asymptotics.  
\item 
The case of a symbol having jump discontinuities 
(e.g. the indicator function 
of a bounded domain in $\R^d$, $d\ge 2$) 
was studied in \cite{Sobolev2013} (smooth $f$ and $\L$) 
and later in \cite{Sobolev2015}, \cite{Sobolev2016} (non-smooth $f$ and $\L$). 
In this case the asymptotics  
for the operator \eqref{Dalpha:eq} have a form different 
from \eqref{bd:eq}, and 
their proof requires different methods. 
\item 
In \cite{Sobolev2016a}  the transition between the smooth and discontinuous 
symbol was studied: 
the smooth 
symbol $a$ was supposed to depend on an extra 
``smoothing"
parameter $T>0$ so that 
$a=a_T$ converged to an indicator function as $T\to 0$. 
The obtained asymptotic formula described the behaviour of the trace of 
\eqref{Dalpha:eq} as the two parameters, $\a$ and $T$, independently 
tended to their respective limits: $\a\to\infty$ and $T\to 0$.
%
On the other hand, the results 
of \cite{Sobolev2016a} 
did not cover the case $\a\to\infty$, $T = const$.  
One aim of the current paper is to bridge this gap.
\end{enumerate}

The non-smooth generalizations are partly motivated by 
new applications of the Szeg\H o asymptotics 
in Statistical Physics, 
connected with the entanglement entropy for free fermions 
(EE), see \cite{GioevKlich2006}, 
\cite{Helling2009}, \cite{LeschkeSobolevSpitzer2014}, 
\cite{LeschkeSobolevSpitzer2016} and references therein.  
In particular, the asymptotic trace formula 
for smooth symbols $a$ (i.e. the one in Theorem \ref{main:thm}) 
is used to describe the EE  
at a positive temperature (see \cite{LeschkeSobolevSpitzer2016}) , 
whereas the zero temperature case requires the use of discontinuous symbols 
(see \cite{LeschkeSobolevSpitzer2014}).  
We briefly comment on these applications after the main Theorem 
\ref{main:thm}.

The paper is organized as follows. In Sect. 2 we provide some preliminary 
information and state the main result, followed by a short discussion of 
the applications to the EE. 
It is not so trivial to see that the main asymptotic coefficient 
$\CB_d(a, \p\L; f)$ is finite, if the function $f$ is non-smooth. This point and 
other useful properties of $\CB_d(a, \p\L; f)$ are clarified in Sect. 3. 
In Sect. 4 we collect some known and some new bounds for trace norms of 
Wiener-Hopf operators. 
Among other bounds, 
Sect. 4 contains the crucial trace-norm estimate 
for the operator \eqref{Dalpha:eq} 
with a non-smooth 
function $f$( see \eqref{crelle2:eq}) obtained in \cite{Leschke2016}. 
The bounds of Sect. 4 are instrumental in the proof 
of the ``local" asymptotics for the operator \eqref{Dalpha:eq}, see 
Theorem \ref{loc:thm} in Sect. 5. The local results are put together 
to complete the proof of Theorem \ref{main:thm} in Sect. 6. 
The proof follows the ideas 
of \cite{Leschke2016}, \cite{Sobolev2015}, \cite{Sobolev2016}. 
Specifically, to justify the formula 
\eqref{bd:eq} 
we use the standard method of asymptotic analysis: 
first we 
prove it 
for polynomial functions $f$, 
then ``close" the asymptotics using the estimate \eqref{crelle2:eq} from Sect. 4. 

Throughout the paper we adopt the following convention.
For two non-negative numbers (or functions) 
$X$ and $Y$ depending on some parameters, 
we write $X\lesssim Y$ (or $Y\gtrsim X$) if $X\le C Y$ with 
some positive constant $C$ independent of those parameters. 
For example, $\a \gtrsim 1$ 
means that $\a \ge c$ with some constant $c$, independent 
of $\a$. 
If $X\lesssim Y$ and $X\gtrsim Y$, then we write $X\asymp Y$. 
To avoid possible misunderstanding we often make explicit comments on the nature of 
(implicit) constants in the bounds. 

The notation $B(\bz, R)\subset\R^d$, $\bz\in\R^d$, $R>0$,
is used for the open ball of radius $R$, centred at the point 
$\bz$. The function $\chi_{\bz, R}$
stands for the indicator of the ball $B(\bz, R)$. 

\textbf{Acknowledgements.} 
This paper grew  out of numerous discussions 
with H. Leschke and W. Spitzer, who have author's deepest gratitude. 
A part of this paper was written during several visits of 
the author to the FernUniversit\"at Hagen in 2015-2016.

The author was supported by EPSRC grant EP/J016829/1.

\section{Main results}
 
First we specify conditions on the set $\L$ under which we study the operator \eqref{Dalpha:eq}. 

\subsection{The domains and regions}\label{domains:subsect} 
Assume that $d\ge 2$. 
We say that $\L$ is a basic Lipschitz 
(resp. basic $\plainC{m}$, $m = 1, 2, \dots$) domain, if 
there is a Lipschitz (resp. $\plainC{m}$) 
function $\Phi = \Phi(\hat\bx)$, $\hat\bx\in\R^{d-1}$, such that 
with a suitable choice of Cartesian coordinates $\bx = (\hat\bx, x_d)$, 
the domain $\L$ is the epigraph of the function $\Phi$, i.e. 
\begin{equation}\label{basic_d:eq}
\L = \{\bx\in\R^d: x_d > \Phi(\hat\bx)\}. 
\end{equation} 
We use the notation $\L = \G(\Phi)$. 
The function $\Phi$ is assumed to be globally Lipschitz, i.e. 
the constant 
\begin{equation}\label{MPhi:eq}
M_{\Phi} = \underset{\hat\bx\not=\hat\by }
\sup\ \frac{|\Phi(\hat\bx) - \Phi(\hat\by)|}{|\hat\bx-\hat\by|},
\end{equation}
 is finite. 
 \textit{Throughout the paper, all 
 estimates involving basic Lipschitz domains 
$\L = \G(\Phi)$, are uniform in the number $M_\Phi$}.  
 
A domain (i.e. connected open set) 
is said to be Lipschitz (resp. $\plainC{m}$) if locally 
it coincides with some basic Lipschitz 
(resp. $\plainC{m}$-) domain. 
We call $\L$ a Lipschitz (resp. $\plainC{m}$-) 
region if $\L$ is a union of 
finitely many Lipschitz (resp. $\plainC{m}$-) 
domains such that their closures are pair-wise disjoint. The boundary 
$\p\L$ is said to be a $(d-1)$-dimensional Lipschitz surface. 

A basic Lipschitz 
domain $\L=\G(\Phi)$ is said to be piece-wise $\plainC{m}$  
with some $m = 1, 2, \dots$, if  the function 
$\Phi$ is $\plainC{m}$-smooth away from a collection of 
finitely many  $(d-2)$-dimensional 
Lipschitz surfaces $L_1, L_2, \dots\subset \R^{d-1}$. 
We denote
\begin{align}\label{pls:eq}
(\p\L)_{\rm s} = \Phi(L_1)\cup \Phi(L_2)\cup \dots\subset\p\L.
\end{align}
This is the subset  
where the $\plainC{m}$-smoothness of the surface $\p\L$ may break down. 
A piece-wise $\plainC{m}$-region $\L$ and the set $(\p\L)_{\rm s}$ 
for it are defined in the obvious way. 
An expanded version of these definitions 
can be found in \cite{Sobolev2014}, 
\cite{Sobolev2015}, and here we omit the standard details. 


The minimal assumptions on the sets
featuring in this paper are laid out in the following condition. 

\begin{cond}\label{domain:cond}
The set $\L\subset \R^d, d\ge 2$, 
is a Lipschitz region, 
and  either $\L$ or 
	$\R^d\setminus\L$ is bounded.
\end{cond}

Some results, including the main asymptotic formula 
in Theorem \ref{main:thm}, 
require higher smoothness of $\L$.  
Note that if $\L$ is a Lipschitz 
(or $\plainC{m}$-) region, then so is 
the interior of $\R^d\setminus \L$.

\subsection{The main result}

Suppose that $a\in\plainC\infty(\R^d)$ satisfies the condition
\begin{align}\label{abounds:eq}
|\nabla^m a(\bxi)|\lesssim \lu\bxi\ru^{-\b},\ \b > d, 
\end{align}
for all $m = 0, 1, 2, \dots$, with some implicit constants that 
may depend on $m$. Here 
we have used the standard notation $\lu \bxi\ru = \sqrt{1+|\bxi|^2}$.

In order to state the main result we need to introduce the principal  
asymptotic coefficient. For a function $g:\mathbb C\mapsto \mathbb C$
define
 \begin{equation}\label{U:eq}
 	U(s_1, s_2; g) 
 	= \int_0^1 \frac{g\bigl((1-t)s_1 + t s_2\bigr) 
 		- [(1-t)g(s_1) + t g(s_2)]}{t(1-t)} \,dt, s_1, s_2\in\mathbb C. 
 \end{equation} 
 This quantity is well-defined for any H\"older function $g$. 
For $d = 1$ the function $U$ immediately 
defines the asymptotic coefficient:
\begin{align}\label{cb1:eq}
\CB_1(a; g) 
= \frac{1}{8\pi^2}\lim\limits_{\epsilon\downarrow 0}
\int\limits_{\R} \int\limits_{|t|>\epsilon} 
\frac{U\bigl(a(\xi), a(\xi + t); g\bigr)}{t^2}
dt  d\xi.
\end{align}
As explained in the next section, 
for functions $g\in\plainC2(\mathbb R)$ the integral above exists in the 
usual sense. 

As already mentioned previously, our main interest is to include 
less smooth functions in the consideration. Precisely, we are interested in the 
functions satisfying the following condition. 

\begin{cond}\label{f:cond}
	Assume that for some integer $n \ge 1$ the function 
	$f\in\plainC{n}(\R\setminus\{ x_0 \})\cap\plainC{}(\R)$ satisfies the 
	bound 
	\begin{equation}\label{fnorm:eq}
	\1 f\1_n = 
	\max_{0\le k\le n}\sup_{x\not = x_0} |f^{(k)}(x)| |x-x_0|^{-\g+k}<\infty
	\end{equation}
	with some $\g >  0$, 
	and is supported on the interval $[x_0-R, x_0+R]$ with some $R>0$.  
\end{cond}

As shown in \cite{Sobolev2016b}, 
for such functions the principal 
value definition \eqref{cb1:eq} 
becomes necessary if $\g$ is small, 
see Proposition \ref{scales:prop} in the next 
section. We often use the notation
\begin{align}\label{vark:eq}
\varkappa = \min\{\g, 1\},\ \forall \g>0.
\end{align}
For $d\ge 2$ we introduce the functional of $a$, 
defined for every $\be\in\mathbb S^{d-1}$ 
as a principal value integral similar to \eqref{cb1:eq}:
\begin{align}\label{ad:eq}
\CA_d(a, \be; g) 
= \frac{1}{8\pi^2}\lim\limits_{\epsilon\downarrow 0}
\int\limits_{\R^d} \int\limits_{|t|>\epsilon} 
\frac{U\bigl(a(\bxi), a(\bxi + t\be); g\bigr)}{t^2}
dt  d\bxi.
\end{align}
Assuming that $\L$ satisfies Condition \ref{domain:cond}, 
for any continuous function $\varphi$ define
\begin{align}\label{cbd:eq}
\begin{cases}
	\CB_d(a, \varphi; \p\L,  g) := \frac{1}{(2\pi)^{d-1}}
	\int_{\p\L}   \varphi \CA_d(a, \bn_\bx; g) dS_\bx,\\[0.2cm]
	\CB_d(a; \p\L, g) := \CB_d(a, 1; \p\L,  g).
\end{cases}	
	\end{align} 
When it does not cause confusion, 
sometimes some or all variables are omitted from the 
notation and we write, for instance, $\CB_d(a)$, $\CB_d$.

It will be useful to rewrite 
$\CA_d, d\ge 2,$ via $\CB_1$. 
For each unit vector $\be\in\R^d, d\ge 2,$ define the hyperplane 
\begin{equation*}
\Pi_{\be} = \{ \bxi\in\R^d: \be\cdot\bxi = 0 \}.
\end{equation*}
Introduce the orthogonal coordinates 
$\bxi = (\hat{\bxi}, t)$ such that 
$\hat{\bxi}\in \Pi_\be, t\in\R$. 
Then, thinking of the symbol $a(\bxi)$ as depending on the real variable $t$, 
 and on the parameter $\hat\bxi$, we can rewrite 
 the definition \eqref{ad:eq} as follows:
 \begin{align}\label{cbd1d:eq}
 \CA_d(a, \be; f)
 = \int_{\Pi_{\be}} 
	\CB_1\bigl(a(\hat{\bxi}, \ \cdot\ ); f\bigr)d\hat{\bxi}.
 \end{align}

	The next theorem constitutes the main result of the paper. 

\begin{thm}\label{main:thm}
Suppose that $a\in\plainC\infty(\R^d), d\ge 2$, is a real-valued 
function that 
satisfies \eqref{abounds:eq}
 Assume also that $\L$ is a piece-wise $\plainC1$-region 
 satisfying Condition \ref{domain:cond}. 

Let $X = \{z_1, z_2, \dots, z_N\}\subset \R$, $N <\infty$, be a collection of 
points on the real line. 
Suppose that $f\in\plainC2(\R\setminus X)$ is a function such that 
in a neighbourhood of each point $z\in X$ 
it satisfies the bound
\begin{equation}\label{fmain:eq}
|f^{(k)}(t)|\lesssim |t - z|^{\g-k}, \ k = 0, 1, 2, 
\end{equation} 
with some $\g>0$.  

 If $\b > d\varkappa^{-1}$,
 then 
 the operator $D_\a(a, \L; f)$ is trace-class and 
\begin{align}\label{main:eq}
\lim\limits_{\a\to\infty} \a^{1-d}\tr D_\a(a, \L; f) 
= \CB_d(a, \p\L; f).
\end{align}
The above asymptotics are uniform in symbols $a$ that satisfy \eqref{abounds:eq} with the same 
implicit constants.
\end{thm}

\begin{rem}\label{lin:rem}
Since 
$D_\a(a, \L; g) = 0$ and 
$\CB_d(a, \p\L; g) = 0$ 
for linear functions $g$, 
in the formula \eqref{main:eq} we can always replace $f$ by 
$f+g$ with a linear function $g$ of our choice. This elementary observation 
becomes useful in the proof of Theorem \ref{homolog:thm} below. 
\end{rem}

Theorem \ref{main:thm} has two useful 
corollaries describing the asymptotics 
of $D_\a(\l a, \L; f)$ as $\a\to\infty$ and $\l\to 0, \l >0$. 
The first one is concerned with the asymptotically homogeneous functions $f$.

\begin{thm}\label{homo:thm}
Let the region $\L$ be as in Theorem \ref{main:thm}. 
Suppose that the family of real-valued 
symbols $\{a_0, a_\l\}$, $\l>0$, satisfies \eqref{abounds:eq} with 
some $\b > d\varkappa^{-1}$, 
uniformly in $\l$, and is such that $a_\l\to a$ as $\l\to 0$ pointwise.

Denote $f_0(t) = M |t|^\g$ with some complex $M$ and $\g >0$. 
Suppose that the function $f\in \plainC2(\R\setminus\{0\})$ 
satisfies the condition 
\begin{align}\label{almosthom:eq}
\lim\limits_{t\to 0} |t|^{n-\g}\frac{d^n}{dt^n} 
\big(
f(t) - f_0(t)
\big) = 0,\ n = 0, 1, 2.
\end{align}
Then 
\begin{align}\label{homo:eq}
\lim\limits_{\a\to\infty\, \l\to 0} 
\bigl(\a^{1-d} \l^{-\g}\tr D_\a(\l a_\l, \L; f)\bigr)
= \CB_d(a_0, \p\L; f_0).
\end{align}
\end{thm}

In the next theorem instead of the homogeneous 
function $|t|^\g$ we consider the function 
\begin{align*}
h(t) = -t\log |t|, \ \ t\in\R,
\end{align*}
which still leads to a homogeneous asymptotic behaviour.

\begin{thm}\label{homolog:thm}
Let the region $\L$ be as in Theorem \ref{main:thm}. 
Suppose that the family of real-valued 
symbols $\{a_0, a_\l\}$, $\l>0$, satisfies \eqref{abounds:eq} with 
some $\b > d\varkappa^{-1}$, 
uniformly in $\l$, and is such that $a_\l\to a$ as $\l\to 0$ pointwise. 
 
Suppose that the function $f\in \plainC2(\R\setminus\{0\})$ 
satisfies the condition 
\begin{align}\label{almosthomlog:eq}
\lim\limits_{t\to 0} |t|^{n-1}\frac{d^n}{dt^n} 
\big(
f(t) - h(t)
\big) = 0,\ n = 0, 1, 2.
\end{align}
Then 
\begin{align}\label{homolog:eq}
\lim\limits_{\a\to\infty\, \l\to 0} 
\bigl(\a^{1-d} \l^{-1}\tr D_\a(\l a_\l, \L; f)\bigr)
= \CB_d(a_0, \p\L; h).
\end{align}
\end{thm}

We do not discuss applications of Theorems 
\ref{homo:thm} and \ref{homolog:thm}, but observe nevertheless that 
the entropy functions \eqref{eta_gamma:eq} and \eqref{eta1:eq} 
satisfy the conditions \eqref{almosthom:eq} and \eqref{almosthomlog:eq} 
respectively.

\subsection{Entanglement Entropy}
Here we briefly explain how Theorem \ref{main:thm} applies to the study of 
the entanglement entropy. 
More detailed discussion of the 
subject can be found in \cite{LeschkeSobolevSpitzer2014}, 
\cite{LeschkeSobolevSpitzer2016}, \cite{Leschke2016}. 

We consider the operator \eqref{Dalpha:eq} with the \textit{Fermi symbol}
\begin{equation}\label{positiveT:eq} 
a(\bxi) := a_{T, \mu}(\bxi) 
:= \frac{1}{1+ \exp \frac{h(\bxi) - \mu}{T}}\,,\quad \bxi\in\R^d,
\end{equation}
where $T>0$ is the temperature and $\mu\in\R$ is the chemical potential. 
The function $h\in \plainC\infty(\R^d)$ is the free (one-particle) 
Hamiltonian, and we assume that 
$h(\bxi)\gtrsim |\bxi|^{\b_1}$ as $|\bxi|\to \infty$ with some $\b_1 >0$, 
so that $a$ decays fast at infinity, and 
that $|\nabla^m h(\bxi)|\lesssim \lu \bxi\ru^{\b_2}$, $m= 0, 1, \dots$ with some $\b_2>0$. 
This ensures that \eqref{positiveT:eq} satisfies \eqref{abounds:eq} with an arbitrary 
$\b >0$.
The parameters $T$ and $\mu$  
are fixed.
For the function $f$ we pick the 
\textit{$\g$-R\'enyi entropy function} 
$\eta_\g: \R\mapsto [0, \infty)$ 
defined for all $\g >0$ as follows. If $\g\not = 1$, then
\begin{equation}\label{eta_gamma:eq}
\eta_\gamma(t) := \left\{\begin{array}{ll} 
\frac{1}{1-\gamma} \log\big[t^\gamma + (1-t)^\gamma\big]& 
\mbox{ for }t\in(0,1),\\[0.2cm]
0&\mbox{ for }t\not\in(0,1),\end{array}\right.
\end{equation}
and for $\gamma=1$ (the von Neumann case) it is defined as the limit 
\begin{equation}\label{eta1:eq}
\eta_1(t) := \lim_{\gamma\to1} \eta_\gamma(t) = \left\{\begin{array}{ll} -t \log(t) -(1-t)\log(1-t)& \mbox{ for } t\in(0,1),\\[0.2cm]
0& \mbox{ for }t\not\in(0,1).\end{array}\right.
\end{equation} 
For $\g\not = 1$ the function $\eta_\g$ satisfies condition \eqref{fmain:eq} 
with $\g$ replaced with $\varkappa = \min\{\g, 1\}$, and with $X=\{0, 1\}$. 
The function $\eta_1$ satisfies \eqref{fmain:eq} with an 
arbitrary $\g\in (0, 1)$, 
and the same set $X$. 

For arbitrary $\L\subset\R^d$ 
we define 
\textit{the $\g$-R\'enyi entanglement entropy (EE)} with respect 
to the bipartition 
$\R^d = \L \cup (\R^d\setminus\L)$, as 
\bee \label{def:EE}
\mathrm{H}_\g(\L) =\mathrm{H}_\g(T,\mu; \L) 
= \tr D_1(a_{T,\mu},\L;\eta_\g) 
+ \tr D_1(a_{T,\mu},\R^d\setminus\L;\eta_\g). 
\ene 
These entropies were studied in 
\cite{LeschkeSobolevSpitzer2016}, \cite{Leschke2016}.
In particular, in \cite{Leschke2016} it was shown that for any $T>0$ 
the EE is finite, if $\L$ satisfies Condition \ref{domain:cond}. 
We are interested in the scaling limit 
of the EE, i.e. the limit 
of $\mathrm{H}_\g(\a\L)$ as $\a\to\infty$. 
The next theorem is a direct consequence 
of Theorem \ref{main:thm}:

\begin{thm}\label{EE:thm}
Assume that $\L$ is a piece-wise $\plainC1$-region satisfying 
Condition \ref{domain:cond}.
Let the symbol $a = a_{T, \mu}$ and the functions $\eta_\g, \g>0,$ 
be as defined in \eqref{positiveT:eq} and \eqref{eta_gamma:eq}-\eqref{eta1:eq} 
respectively. 
Then
\begin{align*}
\lim\limits_{\a\to\infty} 
\a^{1-d} \mathrm{H}_\g(\a\L) = 2\CB_d(a_{T, \mu}, \p\L; \eta_\g).
\end{align*}
\end{thm}

This result was stated in \cite[Theorem]{LeschkeSobolevSpitzer2016}, 
but the article \cite{LeschkeSobolevSpitzer2016} contained only a sketch of 
the proof. 

The EE can be also studied for the 
zero temperature, see \cite{LeschkeSobolevSpitzer2014}. In this case 
the Fermi symbol is naturally replaced 
by the indicator function of the region $\{\bxi\in\R^d: h(\bxi)<\mu\}$.

It is worth pointing out that 
it is also instructive to study the behaviour of $\mathrm{H}_\g(T, \mu; \a\L)$ 
as $\a\to\infty$ and $T\to 0 $ simultaneously. 
This study was undertaken in \cite{Leschke2016} (for $d = 1$) and 
\cite{Sobolev2016a} (for arbitrary $d\ge 2$). 
The results of \cite{Leschke2016} 
require $\a T\gtrsim 1$, $\a\to\infty$, so that, in particular, 
$T = const$ is allowed. 
On the contrary, in the paper \cite{Sobolev2016a}, where the multi-dimensional case 
was studied, both the final result and its proof always require that 
$\a\to\infty, T\to 0$. Thus, the results of \cite{Sobolev2016a}, together 
with Theorem \ref{EE:thm}, describe the large-scale asymptotic 
behaviour (i.e. as $\a\to\infty$) for the 
entire range of bounded temperatures (i.e. $T\lesssim 1$) for $d\ge 2$.

\section{Asymptotic coefficient $\CB_d$}\label{ascoeff:sect}

In this section we collect 
some useful properties of the coefficient $\CB_d$ 
in all dimensions $d\ge 1$.

\subsection{Smooth functions $g$.  
Estimates for the coefficient $\CB_d$} 

The following result is a basis for our asymptotic calculations:

 \begin{prop}\label{Widom_82:prop}
 	[see \cite[Theorem 1(a)]{Widom1982}]  
 	Suppose that $a$ is bounded and 
satisfies 
\begin{equation}\label{sobolev:eq}
\iint\frac{|a(\xi_1) - a(\xi_2)|^2}{|\xi_1-\xi_2|^2} d\xi_1 d\xi_2<\infty.
\end{equation}
Let $g$ be analytic on a 
 	neighbourhood of the closed convex hull of 
 	the function $a$. Then the operator $D_\a(a, \R_\pm; g)$ 
 	is trace class and 
 	\begin{equation}\label{Widom_82:eq}
 		\tr D_\a(a, \R_\pm; g)	= \CB_1(a; g).
 	\end{equation}
 \end{prop}
 
In fact the above asymptotics are known to 
hold under weaker conditions on the symbol 
$a$ and function $g$ (see 
\cite{Peller1989}), but Proposition 
\ref{Widom_82:prop} is sufficient for our purposes. 

Now we concentrate on estimates for the coefficient \eqref{cbd:eq}.
As observed in \cite{Widom1982}, 
 if $g$ is twice differentiable, we can integrate by parts in \eqref{U:eq} 
 to obtain the formula
 \begin{equation*}
 	U(s_1, s_2; g) = 
 	(s_1-s_2)^2 \int_0^1 g''\bigl(
 	(1-t)s_1 + t s_2
 	\bigr)
 	\bigl(
 	t\log t + (1-t) \log (1-t)
 	\bigr) dt.
 \end{equation*} 
 Thus, assuming that $g''$ is uniformly bounded, we arrive at the estimate
\begin{align}\label{cad:eq}
|\CA_d(a, \be; g)|
\lesssim \|g''\|_{\plainL\infty}\int\limits_{\R^d} \int\limits_{\R} 
\frac{|a(\bxi) - a(\bxi+t\be)|^2}{t^2}
dt d\bxi.
\end{align} 
For the sake of simplicity, further estimates are stated for 
symbols $a$ satisfying the bounds \eqref{abounds:eq}. 
Unless otherwise stated, all the estimates are uniform in the symbols $a$ satisfying \eqref{abounds:eq} 
with the same implicit constants. 

\begin{lem}
Suppose that 
$g\in\plainC2(\R)$ and $g''$ is bounded.  
Suppose that $a$ satisfies \eqref{abounds:eq} with some $\b > d/2$, $d\ge 2$, 
and 
that $\L$ satisfies Condition \ref{domain:cond}. 
Then 
\begin{align}\label{adest:eq}
|\CA_d(a, \be; g)|\lesssim \|g''\|_{\plainL\infty},
\end{align}
uniformly in $\be\in\mathbb S^{d-1}$, 
and 
\begin{align}\label{bdtwice:eq}
|\CB_d(a, \varphi; \p\L, g)|
\lesssim \|g''\|_{\plainL\infty} \|\varphi\|_{\plainL\infty}
\meas_{d-1}(\p\L\cap\supp\varphi),
\end{align}
for any continuous function $\varphi$. 

If, in addition, $g'$ is uniformly 
bounded and $\b >d$, then for all $\be, \bb\in\mathbb S^{d-1}$, we have 
\begin{align}\label{contca:eq}
|\CA_d(a, \be; g) - \CA_d(a, \bb; g)|
\lesssim (\|g'\|_{\plainL\infty} + 
\|g''\|_{\plainL\infty})|\be-\bb|^\d,
\end{align}
for any $\d\in (0, 1)$, with an implicit constant depending on $\d$. 
\end{lem}

\begin{proof} 
The bound \eqref{bdtwice:eq} follows from \eqref{adest:eq} in view of the definition \eqref{cbd:eq}.
Let us prove \eqref{adest:eq}. 
Let $r\in (0, 1)$, and assume that $|t|\le r$.  
Write the elementary bound 
\begin{align}\label{elem:eq}
|a(\bxi) - a(\bxi+t\be)|\le |t| \max\limits_{|\boldeta-\bxi|\le 1} 
|\nabla a(\boldeta)|\lesssim |t| \lu \bxi\ru^{-\b}.
\end{align}
Thus the right-hand side of \eqref{cad:eq} (with $\|g''\|_{\plainL\infty}$ omitted)
is bounded by 
\begin{align}\label{elem1:eq}
\int\limits_{\R^d} \int\limits_{|t|<1}
&\ \frac{|a(\bxi) - a(\bxi+t\be)|^2}{t^2}
dt d\bxi
+ \int\limits_{\R^d} \int\limits_{|t|>1} 
\frac{|a(\bxi) - a(\bxi+t\be)|^2}{t^2}
dt d\bxi\notag\\[0.2cm]
&\ \lesssim \int\limits_{\R^d} 
\lu \bxi\ru^{-2\b}\,
d\bxi 
+ \int\limits_{\R^d} \lu \bxi\ru^{-2\b}
\int\limits_{|t|>1} \frac{1}{t^2} dt d\bxi\lesssim 1.   
\end{align}
By \eqref{cad:eq} this leads to \eqref{adest:eq}.

Let us prove \eqref{contca:eq}. 
For arbitrary $r\in (0, 1), R>1$, split 
$\CA_d(\be) = \CA_d(a, \be; g)$ into three terms:
\begin{align*}
8\pi^2\CA_d(\be) = \CK_1(\be; r) + \CK_2(\be; r, R) + \CK_3(\be; R), 
\end{align*} 
with
\begin{align*}
\CK_1(\be; r) =  &\ \int\limits_{\R^d} \int\limits_{|t|< r} 
\frac{U\bigl(a(\bxi), a(\bxi + t\be); g\bigr)}{t^2}\,
dt  d\bxi,
\\[0.2cm]
\CK_2(\be; r, R) =  &\ \int\limits_{\R^d} \int\limits_{r<|t|< R} 
\frac{U\bigl(a(\bxi), a(\bxi + t\be); g\bigr)}{t^2}\,
dt  d\bxi,
\\[0.2cm]
\CK_3(\be; R) =  &\ \int\limits_{\R^d} \int\limits_{|t|> R} 
\frac{U\bigl(a(\bxi), a(\bxi + t\be); g\bigr)}{t^2}\, 
dt  d\bxi.
\end{align*} 
Similarly to the first step of the proof,
\begin{align*}
|\CK_1(\be; r)|\lesssim r \|g''\|_{\plainL\infty} \int\limits_{\R^d}
\max\limits_{|\boldeta-\bxi|\le r} 
|\nabla a(\boldeta)|^2 d\bxi  \lesssim r \|g''\|_{\plainL\infty},
\end{align*} 
and 
\begin{align*}
|\CK_3(\be; R)|\lesssim \|g''\|_{\plainL\infty}\int\limits_{\R^d} |a(\bxi)|^2 
\int\limits_{|t|>R}\frac{1}{t^2}\,dt d\bxi
\lesssim \frac{1}{R} \|g''\|_{\plainL\infty}.
\end{align*} 
In order to estimate the middle integral, 
i.e. $\CK_2$, we point out the following elementary  
estimate:
 \begin{equation}\label{approx1:eq}
|U(s_1, s_2; g) - U(r_1, r_2; g)|\lesssim  \|g'\|_{\plainL\infty}
\bigl(|s_1-r_1|^\d + |s_2-r_2|^\d\bigr),\ \forall \d\in (0, 1),
\end{equation}
with an implicit constant depending on $\d$.  
Substituting $s_1=r_1 = a(\bxi)$ and 
$s_2 = a(\bxi + t\be),\ r_2 = a(\bxi+t\bb)$, and using 
\eqref{elem:eq}, we can estimate as follows:
\begin{align*}
|U(s_1, s_2; g) - U(r_1, r_2; g)|
\lesssim &\ \|g'\|_{\plainL\infty} |a(\bxi + t\be) - a(\bxi+t\bb)|^\d\\[0.2cm]
\lesssim &\ \|g'\|_{\plainL\infty} |t|^\d|\be-\bb|^\d \lu \bxi+t\be\ru^{-\b\d}.
\end{align*}
 Taking $\d\in (0, 1)$ such that $\b\d >d$, we obtain
 \begin{align*}
| \CK_2(\be; r, R) - \CK_2(\bb; r, R)|
 \lesssim &\ \|g'\|_{\plainL\infty} |\be-\bb|^\d
 \int\limits_{r<|t|<R}|t|^{\d-2}
 \int\limits_{\R^d}  \lu\bxi+t\be\ru^{-\b\d}\, d\bxi dt\\[0.2cm]
 \lesssim &\ \|g'\|_{\plainL\infty}|\be-\bb|^\d r^{\d-1}.
 \end{align*}
 Collecting the bounds together, we get: 
 \begin{align*}
| \CA_d(\be) - \CA_d(\bb)|
\lesssim &\ |\CK_1(\be; r)| + |\CK_1(\bb; r)|\\[0.2cm] 
&\ + |\CK_3(\be; R)|+ |\CK_3(\bb; R)|  + 
|\CK_2(\be; r, R) - \CK_2(\bb; r, R)|\\[0.2cm]
\lesssim &\ (\|g'\|_{\plainL\infty}+\|g''\|_{\plainL\infty})
(r + R^{-1} + |\be-\bb|^\d r^{\d-1}).
 \end{align*}
 Take $r = |\be-\bb|^\d$, $R^{-1} = |\be-\bb|$, so 
 that the last bracket is bounded by $|\be-\bb|^{\d^2}$. Re-denote $\d^2\mapsto \d$.
The proof of \eqref{contca:eq} is complete. 
\end{proof}

\subsection{Non-smooth test functions} \label{non_smooth:subsect}
For  functions $f$, satisfying Condition 
\ref{f:cond}, 
 the coefficient $\CB_1(a; f)$ 
was studied in \cite{Sobolev2016b}. 
In order to use the results of \cite{Sobolev2016b} we need to recall the notion of 
\textit{multi-scale symbols}. 
Consider a $\plainC\infty$-symbol $a(\bxi)$ 
for which there exist
 positive continuous functions 
$v = v(\bxi)$ and $\tau = \tau(\bxi)$,  
such that
\begin{equation}\label{scales:eq}
|\nabla_{\bxi}^k a(\bxi)|\lesssim 
\tau(\bxi)^{-k} v(\bxi),\ k = 0, 1, \dots,\quad \bxi\in\R^d.
\end{equation} 
It is natural to call $\tau$ the \textit{scale (function)} 
and $v$ the \textit{amplitude (function)}. 
We  refer to symbols 
$a$ satisfying \eqref{scales:eq} as \textit{multi-scale} symbols.
It is convenient to introduce the notation 
\begin{equation}\label{Vsigma:eq}
V_{\s, \rho}(v, \tau)  := \int \frac{v(\bxi)^\s}{\tau(\bxi)^\rho}d\bxi, \ 
\s>0, \rho\in\R.
\end{equation}
Apart from the continuity we often need some extra conditions on the scale 
and the amplitude. First we assume that $\tau$ is globally Lipschitz,
that is,
\begin{equation}\label{Lip:eq}
|\tau(\bxi) - \tau(\boldeta)| \le \nu |\bxi-\boldeta|,\ \ \bxi,\boldeta\in\R^d,
\end{equation}
 with some $\nu>0$. 
By adjusting the implicit constants in \eqref{scales:eq} we may assume that 
$\nu<1$. It is straightforward to check that 
\begin{equation}\label{dve:eq}
(1+\nu)^{-1}\le 
\frac{\tau(\bxi)}{\tau(\boldeta)} 
\le (1-\nu)^{-1},\ \ \boldeta\in B\bigl(\bxi, \tau(\bxi)\bigr).
\end{equation}
Under this assumption on the scale 
$\tau$, the amplitude $v$ is assumed to satisfy the bounds
\begin{equation}\label{w:eq}
\frac{v(\boldeta)}{v(\bxi)}\asymp 1,\ \boldeta\in B\bigl(\bxi, \tau(\bxi)\bigr).
\end{equation}
If $a$ satisfies \eqref{abounds:eq}, then it can be viewed as 
a multi-scale symbol with 
\begin{align}\label{power:eq}
v(\bxi) = \lu \bxi\ru^{-\b},\ \tau(\bxi) = 1,\ \bxi\in\R^d,
\end{align}
so that 
\begin{align*}
V_{\s, \rho}(v, \tau) \asymp 1,\ \forall \s> d\b^{-1}, \forall \rho\in\R.
\end{align*}
For the next statements recall the definition \eqref{vark:eq}.

\begin{prop}\label{scales:prop}[\cite[Theorem 6.1]{Sobolev2016b}]
Suppose that $f$ satisfies Condition \ref{f:cond} with $n = 2$, 
$\g >0$ and some $R > 0$. 
Let the symbol 
$a\in\plainC\infty(\R)$ be a multi-scale symbol. 
Then for any $\s\in (0, \varkappa]$ 
we have 
\begin{equation}\label{coeffscales:eq}
|\CB_1(a; f)|\lesssim \1 f\1_2 R^{\g-\s} V_{\s, 1}(v, \tau), 
\end{equation}
with a constant independent of $f$, 
uniformly in the functions $\tau, v$, 
and the symbol $a$.
\end{prop}

\begin{cor} 
Let the function $f$ be as in Proposition 
\ref{scales:prop}, and let $\L$ satisfy Condition \ref{domain:cond}. 
Let the symbol 
$a\in\plainC\infty(\R^d), d\ge 2$,  
be a real-valued symbol satisfying 
\eqref{abounds:eq} with $\b > ~d\varkappa^{-1}$.
Then the coefficient $\CB_d(a, \varphi; \p\L, f)$ 
in \eqref{cbd:eq} is well-defined. Moreover, 
for any $\s\in (d\b^{-1}, \varkappa]$ 
it satisfies the bound 
\begin{equation}\label{coeffd:eq}
|\CB_d(a,\ \varphi; \p\L, f)|\lesssim \1 f\1_2 
\|\varphi\|_{\plainL\infty}
\meas_{d-1}(\p\L\cap\supp\varphi) 
R^{\g-\s},
\end{equation}
with an implicit constant independent of the functions 
$f$, $\varphi$, and the region $\L$.
\end{cor}

\begin{proof} 
By the definition \eqref{cbd:eq} it suffices to prove 
that
\begin{align*}
|\CA_d(a, \be; f)|\lesssim \1 f\1_2  
R^{\g-\s},
\end{align*}
uniformly in $\be\in\mathbb S^{d-1}$. Choose the coordinates in such a way that 
$\be = (0, \dots, 0, 1)$, and represent $\bxi\in\R^d$ as $\bxi = (\hat\bxi, \xi_d)$. 
Thus by \eqref{cbd1d:eq},  
\begin{align}\label{fl:eq}
\CA_d(a, \be; f) = \int\limits_{\R^{d-1}} 
\CB_1\big(a(\hat\bxi,\ \cdot\ ); f\big) d\hat\bxi.
\end{align}
By \eqref{abounds:eq}, 
the symbol $a(\hat\bxi,\ \cdot\ )$ satisfies \eqref{scales:eq} 
with 
\begin{align*}
v_{\hat\bxi}(t) = (1+|\hat\bxi|^2 + t^2)^{-\frac{\b}{2}},\ \tau(t) = 1,\ 
\forall t\in\R. 
\end{align*}
It is immediate that
\begin{align*}
V_{\s, \rho} \big(v_{\hat\bxi}, \tau\big)
\lesssim \lu \hat\bxi\ru^{-\s\b+1},\ \forall \rho\in\R, 
\end{align*}
and hence, by \eqref{coeffscales:eq} and \eqref{fl:eq},  
\begin{align*}
|\CA_d(a, \be; f)|\lesssim \1 f\1_2 R^{\g-\s}\int\limits_{\R^{d-1}} 
\lu \hat\bxi\ru^{-\s\b+1} 
d\hat\bxi\lesssim \1 f\1_2 R^{\g-\s},
\end{align*}
under the assumption that $\s\b > d$. This gives the required bound. 
\end{proof}
 
Let us also establish the continuity of the asymptotic coefficient 
$\CB_d$ in the functional parameter $a$:

\begin{cor}\label{cbconvf:cor}
Let the function $f$ be as in Proposition 
\ref{scales:prop}, and let $\L$ satisfy Condition \ref{domain:cond}. 
Suppose that the family of symbols $\{a_0, a_\l\}$, $\l>0$, satisfies \eqref{abounds:eq} 
with some $\b > ~d\varkappa^{-1}$, 
uniformly in $\l$, and is such that $a_\l\to a$ as $\l\to 0$ pointwise. 
Then 
\begin{align}\label{cbconvf:eq}
\CB_d(a_\l,\varphi; \p\L, f)\to \CB_d(a_0,\varphi; \p\L, f),\ \l\to 0.
\end{align}
\end{cor}

\begin{proof} 
Let us consider first a test function $g\in\plainC2(\R)$ with uniformly bounded $g'$ and $g''$, and prove that 
\begin{align}\label{cbconv:eq}
\CB_d(a_\l,\varphi; \p\L, g)\to \CB_d(a_0,\varphi; \p\L, g),\ \l\to 0.
\end{align}
In view of the definition \eqref{cbd:eq} it suffices to prove that  
\begin{align}\label{alconv:eq}
\CA_d(a_\l, \be; g)\to \CA_d(a_0, \be; g),\ \l\to 0,
\end{align}
for each $\be\in\mathbb S^{d-1}$. Indeed, by \eqref{adest:eq} the integrals $\CA_d(a_\l, \be; g)$ 
are bounded uniformly in $\be$, so the Dominated Convergence Theorem would lead to \eqref{cbconv:eq}. 

Proof of \eqref{alconv:eq}. 
According to the bounds \eqref{elem:eq}, \eqref{elem1:eq}, the family  
\begin{align*}
F_\l(\bxi, t ) : = U\big(a_\l(\bxi), a_\l(\bxi+t\be); g\big)
\end{align*}
 has an integrable majorant.  
Furthermore, in view of \eqref{approx1:eq}, 
\begin{align*}
|F_\l(\bxi, t) - F_0(\bxi, t)|
\lesssim \|g'\|_{\plainL\infty}\big(
|a_\l(\bxi)-a_0(\bxi)|^\d + 
|a_\l(\bxi + t\be)-a_0(\bxi+t\be)|^\d
\big).
\end{align*}
Since the right-hand side tends zero as $\l\to 0$, we have the convergence  
$F_\l(\bxi, t)\to F_0(\bxi, t), \l\to 0,$ for all $\bxi, t$.
By the Dominated Convergence Theorem,  \eqref{alconv:eq} holds, as claimed. 

Return to the function $f$. 
Let $\z\in\plainC\infty_0(\R)$ be a real-valued function, such that 
$\z(t) = 1$ for $|t|\le 1/2$. 
Represent $f= f_R^{(1)}+ f_R^{(2)}, 0<R\le 1$, where 
$f_R^{(1)}(t) = f(t) \z\bigl(tR^{-1}\bigr)$,\ 
$f_R^{(2)}(t) = f(t)  - f_R^{(1)}(t)$. 
It is clear that  $f_R^{(2)}\in \plainC2(\R)$, and hence the convergence \eqref{cbconv:eq} holds with 
$g = f_R^{(2)}$, for each $R>0$. Furthermore, since $\1 f_R^{(1)}\1_2\lesssim \1 f\1_2$,  
the bound \eqref{coeffd:eq} implies that 
\begin{align*}
|\CB_d(a_\l, \varphi; \p\L, f_R^{(1)})|\lesssim \1 f\1_2 \|\varphi\|_{\plainL\infty} 
\meas_{d-1}(\p\L\cap \supp\varphi) R^{\g-\s},
\end{align*}
with an arbitrary $\s\in (d\b^{-1}, \varkappa]$. Since $R>0$ is arbitrary, this implies the convergence 
\eqref{cbconvf:eq}. 
\end{proof}

\section{ Estimates for multidimensional Wiener-Hopf operators} 

As always, we assume that $a\in\plainC\infty(\R^d)$ satisfies \eqref{abounds:eq}. 
Our main objective in this section is to prepare some trace-class 
bounds for localized operators, such as $\chi_{\bz, \ell} D_\a(a, \L; g_p)$, where 
$g_p(t) = t^p, p = 1, 2, \dots$. The obtained bounds are uniform in $\bz\in\R^d$, 
and in the symbols $a$ satisfying \eqref{abounds:eq} with the same implicit constants.

As we have noted previously, 
the symbols satisfying \eqref{abounds:eq}, 
can be interpreted as multi-scale symbols 
(see Subsection \ref{non_smooth:subsect}) with the amplitude 
$v=v(\bxi)$ and the scaling function 
$\tau=\tau(\bxi)$ defined in \eqref{power:eq}.  
The bounds in the next proposition are borrowed from 
\cite[Lemma 3.4 and Theorem 3.5]{Leschke2016}, 
where they were obtained for more general multi-scale symbols.
Below we state them for the case \eqref{power:eq} only. 

\begin{prop}\label{Crelle:prop}
Let $a$ be a symbol satisfying \eqref{abounds:eq} 
with some $\b >d$. 
Suppose that $\L$ is a Lipschitz region, 
and that $\a\ell \gtrsim 1$. 
Then  
\begin{align}\label{crelle1:eq}
\|\chi_\L \chi_{\bz, \ell}
\op_\a(a)(I-\chi_\L)\|_1
\lesssim (\a\ell)^{d-1}.
\end{align}
If $\L$ is basic Lipschitz, then this bound is uniform in $\L$.

Suppose in addition that 

-- $\L$ satisfies Condition \ref{domain:cond},

-- the function $f$ satisfies Condition \ref{f:cond} with some $\g>0$, $R>0$  
and  $n=2$,  

-- $\b > d\varkappa^{-1}$,\ where $\varkappa = \min\{\g, 1\}$.

\noindent 
Then for any $\s \in (d\b^{-1}, \varkappa)$ and all $\a\gtrsim 1$ 
we have 
\begin{align}\label{crelle2:eq}
\|D_\a(a, \L; f)\|_1\lesssim  \a^{d-1}\1 f\1_2 R^{\g-\s}.
\end{align}
The implicit constants in 
\eqref{crelle1:eq} and \eqref{crelle2:eq} 
do not depend on $\a$, $f$ and $R$,
but depend on the region $\L$.
\end{prop}
 
The next Proposition is a direct consequence of 
\cite[Lemma 5.2]{Sobolev2016a}, for the symbols satisfying \eqref{abounds:eq}.

\begin{prop}\label{comm:prop} 
Let the symbol $a$ satisfy \eqref{abounds:eq} with $\b >d$. 
Let $\a>0$ and $\ell >0$. 
Then for any $r>1$ and any $m\ge d+1$, we have 
	\begin{equation}\label{decouple:eq}
	\|\chi_{\bz, \ell} 
	\op_{\a}(a)\bigl(1- \chi_{\bz, r\ell}\bigr)\|_1
	\lesssim (\a \ell)^{d-m}, 
	\end{equation} 
	with an implicit constant depending on $r$. 
\end{prop}
 
\begin{lem}\label{localal:lem} 
Let $\L$ be a Lipschitz region, and let 
$\a\ell\gtrsim 1$. Suppose that 
$a\in\plainC\infty(\R^d)$ satisfies \eqref{abounds:eq} with $\b > d$. Then we have 
\begin{align}\label{localal:eq}
\|\chi_{\bz, \ell} D_\a(a, \L; g_p)\|_1
\lesssim (\a\ell)^{d-1}.
\end{align}
\end{lem}

\begin{proof}
The proof is by induction. 
First observe that $D_\a(a, \L; g_1) = 0$, so  
\eqref{localal:eq} trivially holds. 

Suppose that \eqref{localal:eq} holds for some $p=k$. In order to 
prove it for $p = k+1$, write:
\begin{align*}
D_\a(a; g_{k+1})
= &\ D_\a(a; g_k) W_\a(a) 
+ W_\a(a^k) W_\a(a) - W_\a(a^{k+1})\\[0.2cm]
= &\ D_\a(a; g_k) W_\a(a) 
- \chi_\L \op_\a(a^k)(I-\chi_\L) \op_\a(a) \chi_\L.
\end{align*}
Thus by the triangle inequality, 
\begin{align*}
\|\chi_{\bz, \ell}D_\a(a; g_{k+1})\|_1
\le &\ \|\chi_{\bz, \ell} D_\a(a; g_k)\|_1 \|W_\a(a)\| 
+ \|\chi_{\bz, \ell}\chi_\L \op_\a(a^k)(I-\chi_\L)\|_1 \|\op_\a(a)\|\\[0.2cm]
\lesssim &\ (\a\ell)^{d-1},
\end{align*}
where we have used the induction assumption, the bound 
\eqref{crelle1:eq} and the elementary estimate $\|\op_\a(a)\|\lesssim 1$. 
This completes the proof. 
\end{proof}

For any $R>0$ and $p\in\mathbb N$ define the 
$(p+1)$-tuple of numbers
 \begin{align}\label{rR:eq}
 r_j = r_j(R) = R\bigg(1+\frac{j}{p}\bigg), 
 j = 0, 1, 2, \dots, p, 
 \end{align}
so that $r_0 = R$, $r_{p} = 2R$.
 Denote
 \begin{align}
 T_p(a; \L; \bz, R) = &\ \chi_{\bz, R}\prod_{j=1}^p  
 W_\a(a;  B(\bz, r_j)\cap \L),\label{tdef:eq}\\[0.2cm]
 S_p(a; \L; \bz, R) = &\ \big(1-\chi_{\bz, 2R}\big)\prod_{j=1}^p  
 W_\a\big(a; (B(\bz, r_{p-j}))^c\cap\L\big).\label{sdef:eq}
 \end{align}
When it does not cause confusion, 
sometimes we omit the dependence of these operators on some or all variables and write, e.g., 
$T_p(\L), S_p(\L)$ or $T_p, S_p$.

\begin{lem}\label{ta:lem}
Let $\a>0$ and $\ell >0$. Then for any 
$m\ge d+1$, 
\begin{align}
\big\| \chi_{\bz, \ell} g_p(W_\a(a; \L) )
- T_p(a;\L; \bz, \ell)\big\|_1
\lesssim (\a\ell)^{d-m}, 
\label{ta:eq}
\\[0.2cm]
\big\| (I-\chi_{\bz, 2\ell}) g_p(W_\a(a;\L) )
- S_p(a;\L; \bz, \ell)\big\|_1 
\lesssim (\a\ell)^{d-m}.
\label{sa:eq}
\end{align} 
\end{lem}

\begin{proof}
Denote 
\begin{align*}
G_p = \chi_{\bz, \ell}g_p(W_\a(a;\L)),\ T_p = T_p(a; \L; \bz, \ell).
\end{align*}
The proof is by induction. By definition,
\begin{align*}
G_1 - T_1 =  \chi_{\bz, \ell}\chi_\L\op_\a(a) 
\big(I-\chi_{\bz, r_1}\big)\chi_\L, r_1 = r_1(\ell).
\end{align*}
Since $r_1 > \ell$, 
by \eqref{decouple:eq}, the required bound 
\eqref{ta:eq} holds for $p=1$. 
Suppose it holds for some $p = k\ge 1$, 
and let us derive it for $p = k+1$:
\begin{align*}
G_{k+1} - T_{k+1}
= &\ (G_{k}- T_k) W_\a(a; \L)\\[0.2cm]
&\ + T_k \chi_\L\big(
\chi_{\bz, r_k}\op_\a(a) - \chi_{\bz, r_k}\op_\a(a)\chi_{\bz, r_{k+1}}
\big) \chi_\L.
\end{align*}
The last bracket equals
\begin{align*}
\chi_{\bz, r_k}\op_\a(a)\big(I - \chi_{\bz, r_{k+1}}\big),
\end{align*}
so, using for the last term \eqref{decouple:eq} again, we get 
\begin{align*}
\|G_{k+1} - T_{k+1}\|_1
\lesssim &\ 
\|G_k-T_k\|_1 \|W_\a(a; \L)\| 
+ \|T_k\|\ \|\chi_{\bz, r_k}\op_\a(a)\big(I - \chi_{\bz, r_{k+1}}\big)\|_1\\[0.2cm]
\lesssim &\ (\a\ell)^{d-m}, 
\end{align*}
which implies \eqref{ta:eq} for $p=k+1$, as required. Thus, by induction, 
\eqref{ta:eq} holds for all $p = 1, 2, \dots$. 

The bound \eqref{sa:eq} is derived in the same way up to obvious modifications. 
\end{proof}

\begin{cor}\label{local:cor}
Suppose that for some sets $\L$ and $\Pi$ we have 
\begin{align}\label{local:eq}
\L\cap B(\bz, 2\ell) = \Pi\cap B(\bz, 2\ell).
\end{align}
Then for any $m\ge d+1$, and any $\a>0$, $\ell>0$, we have 
\begin{align*}
\|\chi_{\bz, \ell} \bigl(g_p(W_\a(a, \L) ) - 
g_p(W_\a(a, \Pi) )\bigr)\|_1\lesssim 
(\a\ell)^{d-m}.
\end{align*}
\end{cor}

\begin{proof}
Due to the condition \eqref{local:eq}, 
and to the definition \eqref{tdef:eq}, 
we have $T_p(a; \L; \bz, \ell) = T_p(a; \Pi; \bz, \ell)$.  Now 
the required bound follows from \eqref{ta:eq} used first for $\L$ and then for 
$\Pi$.
\end{proof}

\begin{cor}\label{localout:cor}
Suppose that for some sets $\L$ and $\Pi$ we have 
\begin{align}\label{localout:eq}
\L\cap (B(\bz, \ell))^c = \Pi\cap (B(\bz, \ell))^c.
\end{align} 
Then for any $m\ge d+1$, , and any $\a>0$, $\ell>0$ , we have 
\begin{align*}
\|(1-\chi_{\bz, 2\ell}) \bigl(g_p(W_\a(a,\L) ) - 
 g_p(W_\a(a, \Pi) )\bigr)\|_1\lesssim 
(\a\ell)^{d-m}.
\end{align*}
\end{cor}

\begin{proof}
Due to the condition \eqref{local:eq}, 
and to the definition \eqref{sdef:eq}, 
we have $S_p(a; \L; \bz, \ell) = S_p(a; \Pi; \bz, \ell)$.  Now 
the required bound follows from \eqref{sa:eq} 
used first for $\L$ and then for $\Pi$.
\end{proof}

 \begin{lem}\label{2ell_away:lem}
For some set $\L\subset \R^d$ and some $\bz\in\R^d$ 
suppose that $B(\bz, 2\ell)\subset \L$. Then 
 for any $m\ge d+1$, and any $\a>0$, $\ell>0$, we have 
 \begin{align}\label{suz:eq}
 \| \chi_{\bz, \ell} D_\a(a, \L; g_p)\|_1\lesssim 
 (\a\ell)^{d-m}.
 \end{align}

 Suppose that $(B(\bz, \ell))^c\subset \L$. Then 
 \begin{align}\label{outsuz:eq}
 \| (I-\chi_{\bz, 2\ell}) D_\a(a, \L; g_p)\|_1\lesssim 
 (\a\ell)^{d-m}.
 \end{align}
 
 \end{lem}
 
\begin{proof} 
Assume that $B(\bz, 2\ell)\subset \L$. 
By Corollary \ref{local:cor}, 
\begin{align*}
\|\chi_{\bz, \ell}\big(
g_p(W_\a(a, \L)) - g_p(W_\a(a, \R^d))
\big)\|_1
\lesssim &\ (\a\ell)^{d-m},\\[0.2cm]
\|\chi_{\bz, \ell}\big(
W_\a(g_p\circ a, \L) - W_\a(g_p\circ a, \R^d)
\big)\|_1
\lesssim &\ (\a\ell)^{d-m}
\end{align*}
Since $g_p(W_\a(a; \R^d)) = \op_\a(g_p(a)) 
= W_\a(g_p\circ a, \R^d)$, by the 
definition \eqref{Dalpha:eq}, the bounds above imply  
\eqref{suz:eq}. 
The estimate \eqref{outsuz:eq} is proved in the same way.
\end{proof}

Let us establish a variant of Corollary \ref{local:cor} 
without the condition \eqref{local:eq}. 

\begin{lem}\label{tlp:lem} 
Let $\L$ and $\Pi$ be arbitrary 
(measurable) sets. Then for any $m\ge d+1$, , and any $\a>0$, $\ell>0$, we have
\begin{align}\label{tlp:eq}
\|\chi_{\bz, \ell}  \bigl(g_p(W_\a(a, \L) ) - 
&\  g_p(W_\a(a, \Pi) )\bigr)\|_1\notag\\[0.2cm]
\lesssim &\ 
(\a\ell)^{d-m} 
+ \a^d \ell^{\frac{d}{2}}
\meas_d\big(
B(\bz, 2\ell)\cap(\Pi\triangle\L)
\big)^{\frac{1}{2}}.
\end{align} 
\end{lem} 
 
\begin{proof}
By Lemma \ref{ta:lem}, it suffices to show that 
\begin{align}\label{tp:eq}
\|T_p(a, \L; \bz, \ell) - T_p(a, \Pi; \bz, \ell)\|_{1}
\lesssim 
\a^d \ell^{\frac{d}{2}}
\meas_d\big(B(\bz, 2\ell)\cap{(\Pi\triangle\L)}\big)^{\frac{1}{2}}.
\end{align} 
Denote $V = \op_\a(a)$, and let $r_j = r_j(\ell), j = 0, 1, \dots p$ 
be as defined in \eqref{rR:eq}. Estimate for each $j = 1, 2, \dots, p$: 
\begin{align}\label{p1:eq}
\|\chi_{\bz, r_j}(\chi_\L V\chi_\L - \chi_\Pi V\chi_\Pi)\chi_{\bz, r_j}\|_1
\le &\ \|\chi_{\bz, r_j}\chi_{\L\triangle\Pi}V \chi_{\bz, r_j}\|_1
+ \|\chi_{\bz, r_j}V\chi_{\L\triangle\Pi} \chi_{\bz, r_j}\|_1 
\notag\\[0.2cm]
\le &\ 2\big\|\chi_{\bz, r_j}\chi_{\L\triangle\Pi}\op_\a(\sqrt{|a|})\big\|_2 
\ \big\|\op_\a(\sqrt{|a|})\chi_{\bz, r_j}\big\|_2 \notag\\[0.2cm]
\lesssim &\ \a^{d} \ell^{\frac{d}{2}}
\meas_d\big(B(\bz, 2\ell)\cap (\L\triangle\Pi)\big)^{\frac{1}{2}}.
\end{align}
This means that \eqref{tp:eq} holds for $p=1$. 
Assume that \eqref{tp:eq} holds for some $p = k, 1\le k \le p-1$, and let us prove it 
for $p = k+1$.  
Denoting $T_p(\L) = T_p(a, \L; \bz, \ell)$, write:
\begin{align*}
T_{k+1}(\L) - &\ T_{k+1}(\Pi)\\[0.2cm]
= &\ \big(T_k(\L) - T_k(\Pi)\big) 
\chi_{\bz, r_{k+1}}\chi_\L V\chi_{\bz, r_{k+1}}\chi_\L  
+ T_k(\Pi)\chi_{\bz, r_{k+1}}
\big(\chi_\L V\chi_{\L} - \chi_\Pi V \chi_\Pi \big)\chi_{\bz, r_{k+1}}.
\end{align*}
Therefore 
\begin{align*}
\|T_{k+1}(\L) - &\ T_{k+1}(\Pi)\|_1\\[0.2cm]
= &\ \|T_k(\L) - T_k(\Pi)\|_1 \|V\|  
+ \|V\|^k \|\chi_{\bz, r_{k+1}}
\big(\chi_\L V\chi_{\L} - \chi_\Pi V \chi_\Pi \big)\chi_{\bz, r_{k+1}}\|_1
\end{align*}
Now, by the inductive assumption and by \eqref{p1:eq}, we get 
\eqref{tp:eq} for $p=k+1$, and hence \eqref{tlp:eq} holds.
\end{proof} 

In the next section we use  Lemma \ref{tlp:lem} with a very specific choice of the 
domains $\L$ and $\Pi$, which is described below.  
Let $\L$ be a basic Lipschitz domain
$\L = \G(\Phi)$, $\Phi\in\plainC1$. 
Let us fix a point $\hat\bz\in\R^d$ and define the new domain 
\begin{align}\label{lambda0:eq}
\L_0 = \G(\Phi_0),\ 
\Phi_0 (\hat{\bx}) = \Phi(\hat\bz) 
	+ (\hat{\bx} - \hat{\bz})\cdot\nabla\Phi(\hat\bz).  
\end{align}
Thus $\L_0$ is the epigraph of the hyperplane tangent to $\L$ 
at the point $\big(\hat\bz, \Phi(\hat\bz)\big)$.   
Let    
\begin{equation}\label{modcont:eq}
	\varepsilon(s) = \underset{\hat{\bx}, \hat{\bz}:|\hat{\bx} - \hat{\bz}|\le s}
	\max|\nabla\Phi(\hat\bx) - \nabla\Phi(\hat{\bz})|\to 0, s\to 0, 
\end{equation}
be the modulus of continuity of $\nabla\Phi$, so that     
\begin{equation*}
	\underset{|\hat\bx-\hat{\bz}|\le s}\max|\Phi(\hat\bx) - \Phi_0(\hat{\bx})|\le \varepsilon(s)s.
\end{equation*}

\begin{lem}\label{tan:lem}
Let $\L$ and $\L_0$ be as defined above.  
Let $\ell \asymp k\a^{-1}$ with come $k>0$. Then for any $m\ge d+1$, and any $\a>0$, we have 
\begin{align*}
\|\chi_{\bz, \ell}\big(
D_\a(a, \L; g_p) - D_\a(a, \L_0; g_p)
\big)\|_1
\lesssim   \big(k^{d-m} + k^{d} \sqrt{\varepsilon(2\ell)}\big).
\end{align*}
\end{lem}

\begin{proof} Using the definition \eqref{Dalpha:eq}, rewrite
\begin{align*}
D_\a(a, \L; g_p) = g_p\big(W_\a(a,\L)\big) - g_1\big(W_\a(g_p(a), \L)\big). 
\end{align*}
We use Lemma \ref{tlp:lem} with 
$\Pi = \L_0$ and $\ell \asymp k\a^{-1}$, 
first for the difference $g_p\big(W_\a(a,\L)\big) - g_p\big(W_\a(a,\L_0)\big)$, 
and then for $g_1\big(W_\a(g_p(a), \L)\big) - g_1\big(W_\a(g_p(a), \L_0)\big)$. 
Estimate:
\begin{align*}
\meas_d\big(B(\bz, 2\ell) \cap(\L\triangle\L_0)\big)\lesssim \ell^d \varepsilon(2\ell)\lesssim 
k^d\a^{-d}\varepsilon(2\ell).
\end{align*} 
Substituting 
this bound in the estimate \eqref{tlp:eq}, we get the proclaimed result. 
\end{proof}

\section{A partition of unity. Local asymptotics}
  
In this Section we focus on the local asymptotics for basic domains, 
that is we study the trace $\tr \varphi D_\a(a; \L, g_p)$ 
for $\varphi\in\plainC\infty_0(\R^d)$  and a basic $\plainC1$-domain $\L$.

\subsection{A partition of unity. Preliminary bounds}   
For the time being we only assume that $\L = \G(\Phi)$ with a Lipschitz function 
$\Phi$. Under this assumption  
we make use of a partition of unity associated with 
the following scaling function:
\begin{equation}\label{ell:eq}
\ell(\bx) = \ell^{(\varkappa)}(\bx)   
= \frac{1}{8\lu M\ru}\sqrt{(x_d - \Phi(\hat{\bx}))^2 + \varkappa^2}, 
\end{equation}
with some $\varkappa\ge 0$, and with the number 
$M = M_{\Phi}$ defined in \eqref{MPhi:eq}. 
Clearly, $|\nabla\ell| \le 8^{-1}$. 
Therefore the function $\tau = \ell$ satisfies 
\eqref{Lip:eq}, and hence \eqref{dve:eq} is also satisfied:
\begin{align}\label{dvel:eq}
\frac{8}{9}\le \frac{\ell(\boldeta)}{\ell(\bxi)}\le \frac{8}{7},\ 
\boldeta\in B(\bxi, \ell(\bxi)).
\end{align}
The bound $|\nabla\ell|\le 8^{-1}$ also allows us to associate 
with the function \eqref{ell:eq} a Whitney type 
partition of 
unity. The next proposition 
follows directly from \cite[Theorem 1.4.10]{Hoermander1993}.
 
\begin{prop}\label{hor:prop}
Let $\ell=\ell^{(\varkappa)}$ be as defined in \eqref{ell:eq}. 
Then 
one can find a sequence 
$\{\bx_j\}_{j\in\mathbb N}\subset \R^d$ such that the balls 
$B_j = B(\bx_j, \ell_j),\ \ell_j = \ell(\bx_j),$ 
form a covering of $\R^d$ for which the number of intersections is bounded 
by a constant depending only on the dimension $d$ (and not on $\varkappa$). Moreover, 
there exists a (non-negative) partition of unity $\psi_j\in\plainC\infty_0(B_j)$, 
such that  
\begin{align*}
|\nabla^m\psi_j(\bx)|\lesssim \ell_j^{-m},
\end{align*}
for each $m = 0, 1, \dots$, uniformly in $j = 1, 2, \dots$. Furthermore, 
the implicit constants in these bounds are uniform in $\varkappa\ge 0$. 
\end{prop}

For a set $\Om\subset \R^d$ introduce two disjoint groups of indices, 
parametrized by the number $\varkappa>0$:
\begin{align}\label{sigma:eq}
\begin{cases}
\Sigma_1(\Om) = \Sigma_1^{(\varkappa)}(\Om) = &\{ j\in\mathbb N: 
B(\bx_j, 2\ell_j)\cap \p\L\not = \varnothing,\ 
B(\bx_j, \ell_j)\cap\Om\not=\varnothing\},\\[0.2cm] 
\Sigma_2(\Om) = 
\Sigma_2^{(\varkappa)}(\Om) = &\{j\in\mathbb N: B(\bx_j, 2\ell_j)\cap \p\L = \varnothing, 
\ B(\bx_j, \ell_j)\cap\Om\not=\varnothing\}. 
\end{cases}
\end{align}
Note the following useful inequalities.

\begin{lem}\label{sigma2:lem}
Let $\bx\in B_j = B(\bx_j, \ell_j)$ with some $j = 1, 2, \dots$. 
If $j\in \Sigma_1(\R^d)$, then 
\begin{align}\label{sigmain:eq}
|x_d - \Phi(\hat\bx)|\lesssim \varkappa.
\end{align}
If $j\in \Sigma_2(\R^d)$, then 
\begin{align}\label{sigmaout:eq}
|x_d - \Phi(\hat\bx)|\gtrsim \varkappa.
\end{align}
The implicit constants in both bounds may depend only on $M$. 
\end{lem}

\begin{proof} 
First observe that 
\begin{align}\label{dist:eq}
\frac{1}{\lu M\ru}|x_d - \Phi(\hat\bx)|
\le \dist(\bx, \p\L)\le |x_d - \Phi(\hat\bx)|.
\end{align}
Now, by \eqref{dvel:eq}, 
for every $\bx\in B_j, j\in\Sigma_1(\R^d),$ we have 
\begin{align*}
\dist(\bx, \p\L)\le 3\ell_j\le \frac{24}{7}\ell(\bx).
\end{align*}
Together with the left inequality \eqref{dist:eq}, this implies that 
\begin{align*}
|x_d - \Phi(\hat\bx)|
\le \frac{3}{7} \sqrt{|x_d - \Phi(\hat\bx)|^2 + \varkappa^2}, 
\end{align*}
whence \eqref{sigmain:eq}.

If $j\in \Sigma_2(\R^d)$, then by \eqref{dvel:eq} again,
\begin{align*}
\dist(\bx, \p\L)\ge \ell_j\ge \frac{8}{9}\ell(\bx). 
\end{align*}
Together with the right inequality \eqref{dist:eq}, this implies that 
\begin{align*}
\frac{1}{9\lu M\ru} \sqrt{|x_d - \Phi(\hat\bx)|^2 + \varkappa^2} 
\le |x_d - \Phi(\hat\bx)|.
\end{align*}
Since $\lu M\ru\ge 1$, this leads to \eqref{sigmaout:eq}. 
\end{proof}

For functions $\psi_j$ found in Proposition 
\ref{hor:prop}, denote also 
\begin{align}\label{inout:eq}
\psi_{\textup{\tiny out}} = \sum\limits_{j\in \Sigma_2(\Om)} \psi_j,\ 
\psi_{\textup{\tiny in}} = \sum\limits_{j\in \Sigma_1(\Om)} \psi_j. 
\end{align}
To avoid cumbersome notation we sometimes do not reflect the dependence 
of $\psi_{\textup{\tiny out}}$ and $\psi_{\textup{\tiny in}}$ 
on the parameter $\varkappa$ and set $\Om$. 
It is often always clear from the context 
which $\varkappa$ and $\Om$ are used. 

\begin{lem} \label{cyl:lem}
Let $\L = \G(\Phi)$ with a Lipschitz function $\Phi$. 
Suppose that $h$ is a Lipschitz function 
with support in the cylinder  
\begin{align*}
\Om_R(\hat\bz) = \{\bx: |\hat\bx-\hat\bz|< R\},
\end{align*}
with some $\hat\bz\in\R^{d-1}$, 
and such that $h(\bx) = 0$ for $\bx\in \p\L$, i.e. 
$h(\hat\bx, \Phi(\hat\bx)) = 0$ for all $\hat\bx\in\R^{d-1}$. 
Suppose that $\a R\gtrsim 1$. 
Then 
\begin{align}\label{cyl:eq}
\|h D_\a(a, \L; g_p)\|_1\lesssim (\a R)^{d-2} (R\|\nabla h\|_{\plainL\infty}). 
 \end{align}
\end{lem}
 
\begin{proof} By rescaling and translation, we may assume that $R = 1$ and that  
$\hat{\bz} = \hat{\mathbf 0}, \Phi(\hat{\mathbf 0}) = 0$. 
Also, 
without loss of generality assume that $|\nabla h|\le 1$, so that 
$|h(\bx)|\le |x_d - \Phi(\hat\bx)|$.  
 
In this proof it is convenient to 
use the function \eqref{ell:eq} with $\varkappa = \a^{-1}$. 
Denote for brevity $\Sigma_m = \Sigma_m^{(\a^{-1})}(\Om_1)$, 
$m = 1, 2$. 
Let $\{\psi_j\}$ be the partition of unity in Proposition \ref{hor:prop}, 
and let $\psi_{\textup{\tiny out}}$ and $\psi_{\textup{\tiny in}}$ 
be the functions defined in \eqref{inout:eq} for $\Om = \Om_1$. 
If $j\in \Sigma_2$,  
we get from Lemma \ref{2ell_away:lem} the following bound: 
\begin{align*}
\|\chi_{B_j} D_\a(a, \L; g_p)\|_1\lesssim (\a\ell_j)^{d-m}, 
\ \forall m\ge d+1.
\end{align*}
In order to collect contributions from all such balls, 
observe that $|h(\bx)|\lesssim \ell_j$ for $\bx\in B_j$, and hence 
\begin{align}\label{out:eq}
\sum\limits_{j\in\Sigma_2} \|h \chi_{B_j} D_\a(a, \L; g_p)\|_1
\lesssim \a^{d-m} \sum\limits_{j\in\Sigma_2} 
\ell_j^{d+1-m}. 
\end{align}
In view of \eqref{dvel:eq}, we can estimate as follows:
\begin{align*}
\ell_j^{d+1-m}\lesssim \int\limits_{B_j\cap\Om_3} 
\ell(\bx)^{d+1-m} d\bx, \ \textup{if}\ \ 
\ell_j\ge 1,\  j\in \Sigma_2,
\end{align*}
and 
\begin{align*}
\ell_j^{d+1-m}\lesssim \int\limits_{B_j\cap\Om_3} \ell(\bx)^{1-m} d\bx, \ \textup{if}\ \ 
\ell_j\le 1,\  j\in\Sigma_2.
\end{align*}
Now we can sum up these inequalities 
remembering that the number of overlapping balls $B_j$ is uniformly bounded:
\begin{align*}
\sum\limits_{j\in\Sigma_2} 
\ell_j^{d+1-m} 
\lesssim &\ \int\limits_{|\hat\bx|\le 3, \ell(\bx)< 1}
\ell(\bx)^{1-m} d\bx 
+ \int\limits_{|\hat\bx|\le 3, \ell(\bx)\ge 1}\ell(\bx)^{d-m+1} d\bx\\[0.2cm]
\lesssim &\ \int\limits_{|t|<1}(|t|+\a^{-1})^{1-m} dt 
+ \int\limits_{|t|\ge 1} |t|^{d-m+1}dt
\\[0.2cm]
\lesssim &\ \a^{m-2},  
\end{align*}
where we have taken $m\ge d+3$ to ensure the 
convergence of the second integral. 
Now it follows from \eqref{out:eq} that  
\begin{align}\label{out1:eq}
\| h \psi_{\textup{\tiny out}}D_\a(a, \L; g_p) \|_1\lesssim \a^{d-2}.
\end{align}

Now consider the indices $j\in \Sigma_1$. By \eqref{sigmain:eq}, 
$\a\ell_j\asymp 1$, and hence 
we get from \eqref{localal:eq} that 
\begin{align*}
\|\chi_{B_j}D_\a(a, \L; g_p)\|_1\lesssim 1.
\end{align*} 
Taking into account that 
$|h(\bx)|\lesssim \a^{-1}$ for $\bx\in B_j$, 
uniformly in $j\in\Sigma_1$, 
and that $\#\Sigma_1\lesssim \a^{d-1}$, we can write:
\begin{align*}
\| h\psi_{\textup{\tiny in}} D_\a(a, \L; g_p)\|_1\lesssim  
\a^{-1} 
\sum\limits_{j\in\Sigma_1}\|\chi_{B_j}D_\a(a, \L; g_p)\|_1\lesssim \a^{d-2}.
\end{align*}
Together with \eqref{out1:eq}, this gives \eqref{cyl:eq}.
\end{proof}

\subsection{Local asymptotics}
Let the coefficient $\CB_1$ and $\CB_d$ 
be as defined in \eqref{cb1:eq} and \eqref{cbd:eq} respectively. 

\begin{lem}\label{planeas:lem}
Let $\Pi \subset\R^d$ be a half-space. 
Suppose that $\varphi\in\plainC\infty_0(\R^d)$  
satisfies the conditions 
\begin{align*}
\ell|\nabla\varphi|\lesssim 1,\ \supp\varphi\subset B(\bz, \ell),
\end{align*}
with some $\bz\in\R^d$ and $\ell >0$ such that $\a\ell\gtrsim 1$. Then 
\begin{align}\label{planeas:eq}
\tr \varphi D_\a(a, \Pi; g_p) 
= \a^{d-1} \CB_d(a, \varphi; \p\Pi, g_p)+ O\big((\a\ell)^{d-2}\big). 
\end{align} 
These asymptotics are 
uniform in the symbols $a$ satisfying \eqref{abounds:eq} with the same implicit constants. 
\end{lem}

\begin{proof}
Without loss of generality assume that 
\begin{align*}
\Pi = \{\bx\in\R^d: x_d>0\}.
\end{align*}
Denote $h(\hat\bx) = \varphi(\hat\bx, 0)$. Since $\varphi - h = 0$ on 
$\p\Pi$, by Lemma \ref{cyl:lem}, we have
\begin{align}\label{planeas1:eq}
\|(\varphi - h)D_\a(a, \Pi; g_p)\|_1\lesssim (\a\ell)^{d-2} 
(\ell \|\nabla\varphi\|_{\plainL\infty}).
\end{align}
The operator $h D_\a$ can be viewed as an 
$\a$-pseudo-differential operator 
in $\plainL2(\R)$ 
with the operator-valued symbol 
\begin{align*}
h(\hat\bx) D_\a\big(a(\hat\bxi,\ \cdot\ ), \R_+; g_p\big).
\end{align*}
Thus its trace is given by the formula 
%
\begin{align*}
\tr h D_\a(a, \Pi; g_p) 
= 
\bigg(\frac{\a}{2\pi}\bigg)^{d-1}
\int\limits_{\R^{d-1}}\int\limits_{\R^{d-1}} \tr 
\bigg(h(\hat\bx) D_\a\big(a(\hat\bxi,\ \cdot\ ), \R_+; g_p)\big)\bigg)
d\hat\bxi d\hat\bx.
\end{align*}
By Proposition \ref{Widom_82:prop}, the trace  
under the integral equals 
\begin{align*}
h(\hat\bx)\CB_1(a(\hat\bxi,\ \cdot\ ), g_p), \ \forall \hat\bxi\in\R^{d-1}, 
\hat\bx\in\R^{d-1},  
\end{align*}
and hence, by \eqref{cbd1d:eq} and \eqref{cbd:eq}, we have the identity 
\begin{align*}
\tr h D_\a(a, \Pi; g_p) 
= \a^{d-1} \CB_d(a, \varphi; \p\Pi, g_p).
\end{align*}
Here we have used the fact that $h = \varphi$ on the hyperplane $\p\Pi$. 
Together with \eqref{planeas1:eq} this gives 
\eqref{planeas:eq}.
\end{proof}

Now we extend the above result to arbitrary $\plainC1$-boundaries. 

\begin{lem} \label{boundas:lem} 
Let $\L$ be a basic $\plainC1$-domain. 
Assume that $\ell\asymp k\a^{-1}$.  
Let $\varphi$ be as in Lemma \ref{planeas:lem}. 
Then 
\begin{align}\label{boundas:eq}
\lim\limits_{k\to\infty}\limsup\limits_{\a\to\infty}\ 
k^{1-d}\bigg|\tr \big(\varphi D_\a(a, \L; g_p)\big)
 - \a^{d-1}\CB_d(a, \varphi; \p\L, g_p)\bigg|
 = 0,
\end{align}
uniformly in $\bz$. The convergence is also uniform in $a$, as in Lemma 
\ref{planeas:lem}.

\end{lem}

\begin{proof} 
For brevity, for $D_\a$ and $\CB_d$ 
 we use the notation omitting the dependence on 
all parameters except $\L$, $\p\L$ and $\varphi$, i.e. 
we write $D_\a(\L)$ and $\CB_d(\varphi; \p\L)$.

For two  
functions $F=F(\a, k)$ and $G=G(\a, k)$ 
we use the notation $F\sim G$ if 
\begin{align*}
\lim_{k\to\infty}\limsup\limits_{\a\to\infty}k^{1-d}(F - G) = 0. 
\end{align*}
Let $\L_0$ be the domain defined in \eqref{lambda0:eq}. 
By Lemma \ref{tan:lem}, for any $m\ge d+1$, we have 
\begin{align*}
|\tr \varphi D_\a(\L)
- \tr \varphi D_\a(\L_0)|\lesssim \big(k^{d-m}
+ k^d \sqrt{\varepsilon(2\ell)}\big).  
\end{align*}
Since $\varepsilon(2\ell)\to 0$ as $\a\to\infty$, for each $k$, 
we conclude that $\tr (\varphi D_\a(\L))\sim \tr (\varphi D_\a(\L_0))$. 
Furthermore, by Lemma \ref{planeas:lem}, 
\begin{align*}
\tr \big(\varphi D_\a(\L_0)\big) 
=  \a^{d-1}\CB_d(\varphi; \p\L_0) + O(k^{d-2}),
\end{align*}
so $\tr (\varphi D_\a(\L_0))\sim \a^{d-1}\CB_d(\varphi; \p\L_0)$. 
Let us now compare the asymptotic coefficients $\CB_d$ for the 
boundaries $\p\L$ and $\p\L_0$, using the definition 
\eqref{cbd:eq} and the bound \eqref{contca:eq}: 
\begin{align*}
|\CB_d(\varphi; \p\L) - \CB_d(\varphi; \p\L_0)|
\lesssim &\ \max |\CA_d(\bn_\bx) - \CA_d(\bn_\bz)| \ell^{d-1}\\[0.2cm]
\lesssim &\ \ell^{d-1}\max |\bn_\bx - \bn_\bz|^{\d},
\end{align*}
where the maximum is taken over $\bx\in \p\L\cap B(\bz, \ell)$, 
and $\d\in (0, 1)$ is arbitrary. 
By \eqref{modcont:eq}, 
\begin{align*}
\max |\bn_\bx - \bn_\bz|\lesssim 
\max|\nabla\Phi(\hat\bx) - \nabla\Phi(\hat\bz)| \le \varepsilon(\ell).
\end{align*}
Consequently, 
\begin{align*}
|\CB_d(\varphi; \p\L) - \CB_d(\varphi; \p\L_0)|
\lesssim \ell^{d-1}\varepsilon(\ell)^{\d}
\lesssim 
k^{d-1} \a^{1-d} 
\varepsilon(\ell)^{\d}, 
\end{align*}
with an arbitrary $\d\in(0, 1)$, and hence 
$ \a^{d-1}\CB_d(\varphi; \p\L) 
\sim \a^{d-1}\CB_d(\varphi; \p\L_0)$. Collecting the equivalence 
relations established above, 
we get $\tr(\varphi D_\a(\L))\sim \a^{d-1}\CB_d(\varphi; \p\L)$, 
which is exactly the formula \eqref{boundas:eq}.
\end{proof}
 
The next step is to extend  Lemma \ref{boundas:lem} 
to the functions $\varphi$ with support of a fixed size, 
i.e. independent of $\a$. 

\begin{thm} \label{loc:thm} 
Let $\L$ be a basic $\plainC1$-domain, and let 
$\varphi\in\plainC\infty_0$. 
Then  
\begin{align}\label{loc:eq}
\tr \big(\varphi D_\a(a; \L, g_p) \big)
= \a^{d-1} \CB_d(a, \varphi; \p\L, g_p) + o(\a^{d-1}), \ \a\to\infty.
\end{align}
The convergence is uniform in $a$, as in Lemma \ref{planeas:lem}.
The remainder depends on the function $\varphi$, and the domain $\L$.  
\end{thm}

\begin{proof} 
Without loss of generality  
we may assume that $\supp\varphi$ is contained in the ball $B = B(\bold0, 1)$. 
Let $\ell = \ell^{(\varkappa)}$ be the function defined in \eqref{ell:eq} 
with $\varkappa = k\a^{-1}$ where $k \ge 1$. 
Let $\{B_j\}$ and 
$\{\psi_j\}$ be the covering of $\R^d$ and the subordinate 
partition of unity a in Proposition \ref{hor:prop} respectively, 
and let $\psi_{\textup{\tiny out}}$ and $\psi_{\textup{\tiny in}}$ 
be as defined in \eqref{inout:eq} with $\Om = B$.
We do 
not reflect in this notation the dependence 
on $k$ and $\a$. 
For brevity we write $D_\a, \CB_d(\psi)$ instead of $D_\a(a, \L; g_p)$ 
and $\CB_d(a, \psi; \p\L, g_p)$.

We consider separately two sets of indices $j$: 
$\Sigma_1(B)$ and $\Sigma_2(B)$, see \eqref{sigma:eq} for the definition. 

\underline{Step 1.} First we handle $\Sigma_2(B)$ and prove that for any $m\ge d+1$ 
the following bound holds: 
\begin{align}\label{sumpsij:eq}
\|\psi_{\textup{\tiny out}}\varphi D_\a\|_1
\lesssim \a^{d-1} k^{-m+1}.  
\end{align}
By definition of $\Sigma_2$, $B(\bx_j, 2\ell_j)\cap\p\L = \varnothing$, so 
by Lemma \ref{2ell_away:lem}, the left-hand side of \eqref{sumpsij:eq} 
does not exceed 
\begin{align*}
\underset{j\in\Sigma_2(B)}\sum \|\psi_j\varphi D_\a\|_1
\lesssim &\ \a^{d-m} \sum\limits_{j\in\Sigma_2(B)} \ell_j^{d-m}
\lesssim \a^{d-m}\sum\limits_{j\in\Sigma_2(B)} 
\int\limits_{B_j}\ell(\bx)^{-m} d\bx\\[0.2cm]
\lesssim &\ 
\a^{d-m} \int\limits_{B(\bold0, 2)} \ell(\bx)^{-m} d\bx 
\lesssim 
\a^{d-m}  
\int\limits_{k\a^{-1}}^1 t^{-m}dt\lesssim \a^{d-1} k^{-m+1},
\end{align*}
for any $m\ge d+1$. As in the proof of Lemma \ref{cyl:lem}, 
when passing from the sums to integrals, we have used 
the property \eqref{dvel:eq}. 
This completes the proof of \eqref{sumpsij:eq}.

\underline{Step 2.} 
Let us now turn to the function $\psi_{\textup{\tiny in}}$. 
At this step we prove that 
\begin{align}\label{kalpha:eq}
\lim\limits_{k\to\infty}\limsup\limits_{\a\to\infty}\big|
\a^{1-d}
\tr \big( \psi_{\textup{\tiny in}}\varphi D_\a\big)
- \CB_d(\varphi)\big| = 0. 
\end{align}
In view of 
\eqref{sigmain:eq}, we have 
$\ell_j\asymp k\a^{-1}$ uniformly in $j\in \Sigma_1(B)$. Thus, by Lemma 
\ref{boundas:lem}, 
\begin{equation}\label{boundastot:eq}
\lim\limits_{k\to\infty}\limsup\limits_{\a\to\infty}k^{1-d} 
 \max\limits_{j\in \Sigma_1(B)} \bigg|
 \tr(\psi_j\varphi D_\a) - \a^{d-1}\CB_d(\psi_j \varphi)
  \bigg| = 0.
  \end{equation} 
  Now we can estimate the left-hand side of \eqref{kalpha:eq}. 
 Since $\#\Sigma_1(B)\lesssim \a^{d-1}k^{1-d}$, 
we have 
\begin{align*}
\big|\a^{1-d} \tr \big( \psi_{\textup{\tiny in}}\varphi D_\a\big)
- \CB_d(\varphi)\big|
=  &\ \a^{1-d}\bigg|\sum\limits_{j\in\Sigma_1(B)}\big(
\tr(\psi_j\varphi D_\a) - \a^{d-1}\CB_d(\psi_j \varphi)
 \big)\bigg|\\[0.2cm]
 \lesssim &\ k^{1-d}
 \max\limits_{j\in\Sigma_1(B)} \bigg|
 \tr(\psi_j\varphi D_\a) - \a^{d-1}\CB_d(\psi_j \varphi)
  \bigg|.
\end{align*} 
By \eqref{boundastot:eq} the double limit (as $\a\to\infty$ and then $k\to\infty$) 
of the right-hand side 
equals zero, which implies \eqref{kalpha:eq}.

\underline{Step 3.} Proof of \eqref{loc:eq}.
According to 
\eqref{sumpsij:eq}, for any $m\ge d+1$, we have 
\begin{align*}
\limsup\limits_{\a\to\infty}
\big|\a^{1-d} \tr\big(\varphi D_\a\big) - &\  \CB_d(\varphi)\big|\\[0.2cm]
\le &\ \limsup\limits_{\a\to\infty}
\big|\a^{1-d} \tr\big(\psi_{\textup{\tiny in}}\varphi D_\a\big) 
- \CB_d(\varphi)\big| + \limsup\limits_{\a\to\infty}
\a^{1-d}\|\psi_{\textup{\tiny out}}\varphi D_\a\|_1\\[0.2cm]
\lesssim &\  \limsup\limits_{\a\to\infty}
\big|\a^{1-d} \tr\big(\psi_{\textup{\tiny in}}\varphi D_\a\big) 
- \CB_d(\varphi)\big| + k^{-m+1}.
\end{align*}
Since $k>0$ is arbitrary, we can pass to the limit as $k\to\infty$, 
so that, by \eqref{kalpha:eq}, the right-hand side tends to zero. This 
leads to  \eqref{loc:eq}, as claimed. 
\end{proof}

\section{Proof of Theorem \ref{main:thm}}
\label{proof2:sect}
 
\subsection{Proof of Theorem \ref{main:thm}: basic piece-wise 
smooth domains $\L$}

Before completing the proof of Theorem \ref{main:thm} 
we extend the formula \eqref{loc:eq} to 
basic piece-wise $\plainC1$-domains. 
 
\begin{thm}\label{pw:thm}
 Let $\L$ be a basic piece-wise $\plainC1$-domain, and let 
 $\varphi\in\plainC\infty_0(\R^d)$. Then the formula \eqref{loc:eq} 
 holds. 
 \end{thm} 

\begin{proof}
As in the proof 
of Theorem \ref{loc:thm}, assume that 
$\varphi$ is supported on the ball $B = B(\bold0, 1)$.  
Further argument follows the proof of \cite[Theorem 4.1]{Sobolev2015}, 
where the asymptotics for $D_\a(a, \L; g_p)$
were studied in the case of a discontinuous symbol $a$.
Thus we give only a ``detailed sketch" of the proof. 

Cover $B$ with open 
balls of radius $\varepsilon>0$, such that the number of intersecting balls  
is bounded from above uniformly in $\varepsilon$. 
Introduce a  subordinate partition of unity $\{\phi_j\}, j = 1, 2, \dots$, 
such that 
\begin{align*}
|\nabla^n \phi_j(\bx)|\lesssim \varepsilon^{-n},\ \forall \bx\in B,
\end{align*}
uniformly in $j= 1, 2, \dots$.   
By Lemma \ref{2ell_away:lem}, the contributions 
to \eqref{loc:eq} from the balls having empty intersection 
with $\p\L$, are of order $O(\a^{d-m})$, $\forall m\ge d+1$, and hence they 
are negligible. 

Let $S$ be the set of indices such 
that the ball indexed by $j\in S$ has a non-empty intersection 
with the set $(\p\L)_{\rm s}$, see \eqref{pls:eq} for the definition. 
Since the set $(\p\L)_{\rm s}$ is built out of $(d-2)$-dimensional Lipschitz surfaces, 
we have 
\begin{equation}\label{sigmaS:eq}
\# S\lesssim \varepsilon^{2-d}.
\end{equation}
If $\a\varepsilon\gtrsim 1$, then by \eqref{localal:eq}, for each $j\in S$ we have the bound 
\begin{equation*}
\|\varphi\phi_jD_\a(a, \L; g_p)\|_1
\lesssim (\a \varepsilon)^{d-1},
\end{equation*}
uniformly in $j$. 
By virtue of \eqref{sigmaS:eq}, this implies that
\begin{equation*}
\sum\limits_{j\in S}
\|\varphi \phi_jD_\a(a, \L; g_p)\|_1
\lesssim \varepsilon \a^{d-1},\ \textup{if}\ \ \a\varepsilon\gtrsim 1.
\end{equation*}
Since
\begin{equation*}
\underset{j\in S}\sum 
\bigl|
\CB_d(a, \varphi\phi_j; \p\L, g_p)
\bigr|\lesssim \varepsilon,
\end{equation*} 
as well, 
we can rewrite the last two formulas as follows:
\begin{align}\label{vare:eq}
\limsup\limits_{\a\to\infty}
\sum\limits_{j\in S} 
\ \biggl|\frac{1}{\a^{d-1}} 
\tr\bigl(\varphi\phi_j D_\a(a, \L; g_p)\bigr)
 - \CB_d(a, \varphi\phi_j; \p\L, g_p)\biggr|
 \lesssim \varepsilon.
\end{align}
Let us now turn to the balls 
with indices $j\notin S$, such that their 
intersection with $\p\L$ is non-empty. 
We may assume that they are separated from 
$(\p\L)_{\rm s}$. Thus in each such ball the boundary of $\L$ is $\plainC1$. 
By Corollary \ref{local:cor}, we may assume that the entire 
$\L$ is $\plainC1$, and hence Theorem  
\ref{loc:thm} is applicable. Together with \eqref{vare:eq}, 
this gives 
\begin{align*}
\limsup\limits_{\a\to\infty}
\biggl|\frac{1}{\a^{d-1}} 
\tr\bigl(\varphi D_\a(a, \L; g_p)\bigr)
 - \CB_d(a, \varphi; \L, g_p)\biggr|
 \lesssim \varepsilon.
\end{align*}
Since $\varepsilon>0$ is arbitrary, this proves the Theorem.
\end{proof}

\subsection{Proof of Theorem \ref{main:thm}: completion}  
Now we can proceed with the proof of Theorem \ref{main:thm}. 
It follows the idea of 
\cite{Sobolev2016} and \cite{Leschke2016}, 
and consists of three parts: 
first we consider polynomial functions $f$, then extend it to arbitrary 
$\plainC2$-functions, and finally complete the proof for functions satisfying the 
conditions of Theorem \ref{main:thm}.

\underline{Step 1. Polynomial $f$.}  
The local asymptotics, i.e. 
Theorem \ref{pw:thm},  
extends to arbitrary piece-wise $\plainC1$-region $\L$ by 
using the standard partition of unity 
argument based on Corollary \ref{local:cor}. 

Now we turn to proving the global asymptotics \eqref{main:eq} for polynomial $f$. 
Let $R_0$ be such that either $\L\subset B(\bold0, R_0)$ or 
$\L^c\subset B(\bold0, R_0)$.  
Let $\varphi\in\plainC\infty_0(\R^d)$ 
be a function such that 
$\varphi(\bx) = 1$ for $|\bx|\le 2R_0$, and $\varphi(\bx) = 0$ 
for $|\bx| > 3R_0$. 
Thus
\begin{align*}
\tr D_\a(a, \L; g_p)
= \tr (\varphi D_\a(a, \L; g_p))
+ \tr \big((1-\varphi) D_\a(a, \L; g_p)\big).
\end{align*}
As we have just observed, by \eqref{loc:eq}, 
the first trace behaves as $\a^{d-1} \CB_d(a, \p\L; g_p)$,  as $\a\to\infty$.
If $\L\subset B(\bold0, R_0)$, 
then the second term equals zero, and hence 
\eqref{main:eq} is proved for $f = g_p$. 

If $\L^c\subset B(\bold0, R_0)$, then, 
by Lemma \ref{2ell_away:lem}, 
the second trace does not exceed 
$\a^{d-m}$ with an arbitrary 
$m\ge d+1$, and hence it gives zero contribution to the formula \eqref{main:eq}. 
Therefore \eqref{main:eq} for $f = g_p$ is proved again. 

\underline{Step 2. Arbitrary functions $f\in\plainC2(\R)$.} 
The extension from polynomials to more general 
functions is done in the same way as in 
\cite{Leschke2016}, and we remind this argument for the sake of 
completeness.  

Since the operator $W_\a(a; \L)$ is bounded uniformly in $\a$, 
we may assume that $f\in \plainC2_0(\R)$, so that $f = f\z$ 
with some fixed function 
$\z\in\plainC\infty_0(\R)$. For a $\d >0$, 
let $g = g_\d$ be a polynomial such that 
\[
\|(f-g)\z\|_{\plainC2} < \d.
\] 
For $g$  we can use the formula \eqref{main:eq} established at Step 1:
\begin{equation}\label{fp:eq} 
\lim\limits_{\a\to\infty}
\a^{1-d}
\tr D_\a(g) = \CB_d(g).
\end{equation}
On the other hand, 
thinking of the function $(f-g)\z$ as satisfying 
Condition \ref{f:cond} with some fixed $x_0$ outside the support of 
$\z$, 
we obtain from 
\eqref{crelle2:eq} that 
\begin{align*}
\|D_\a(f-g)\|_1 = &\ \| D_\a\big((f-g)\z\big)\|_1\\[0.2cm]
\lesssim &\ \1 (f-g)\z\1_2 \a^{d-1} 
\lesssim \|(f-g)\z\|_{\plainC2}\ \a^{d-1}\lesssim \d \a^{d-1},
\end{align*} 
and also, by \eqref{bdtwice:eq}, 
\[
|\CB_d(f) - \CB_d(g)|
= 
|\CB_d(f - g)| = |\CB_d\big((f - g)\z\big)|
\lesssim \| \big((f - g)\z\big)''\|_{\plainL\infty} \lesssim \d.
\]
Thus, using \eqref{fp:eq} and the additivity 
\[
D_\a(f) = D_\a(g) + D_\a(f - g),\ 
\CB_d(f) = \CB_d(g) + \CB_d(f - g),
\] we get
\begin{equation*}
\limsup\limits_{\a\to\infty}
\bigl| \a^{1-d}
\tr D_\a(f) -  \CB_d(f)\bigr|\lesssim \d.
\end{equation*}
Since $\d>0$ is arbitrary, we obtain 
\eqref{main:eq} for arbitrary $f\in\plainC2(\R)$.

\underline{Step 3. Completion of the proof.} 
Let $f$ be a function as specified in Theorem \ref{main:thm}. 
Without loss of generality 
suppose that the set $X$ consists 
of one point, and this point is $z = 0$. 

Let $\z\in\plainC\infty_0(\R)$ be a real-valued function, such that 
$\z(t) = 1$ for $|t|\le 1/2$. 
Represent $f= f_R^{(1)}+ f_R^{(2)}, 0<R\le 1$, where 
$f_R^{(1)}(t) = f(t) \z\bigl(tR^{-1}\bigr)$,\ 
$f_R^{(2)}(t) = f(t)  - f_R^{(1)}(t)$. 
It is clear that  $f_R^{(2)}\in \plainC2(\R)$, 
so one can use the formula 
\eqref{main:eq} established in Step 2 of the proof:
\begin{equation}\label{g2R_le:eq}
 \lim\limits_{\a\to\infty}
 \a^{1-d} D_\a(f_R^{(2)}) =  \CB_d(f_R^{(2)}).
\end{equation}
For $f_R^{(1)}$ we use \eqref{crelle2:eq} taking into account that 
$\1 f_R^{(1)}\1_2\lesssim \1 f\1_2$:
\begin{equation*}
 |\tr D_\a(f_R^{(1)})|
\lesssim 
R^{\g-\s} \1 f\1_2 \a^{d-1},  \ \ \a\gtrsim1,  
\end{equation*}
for any $\s \in (d\b^{-1}, \g)$, $\s\in (0, 1]$.
Moreover, 
by \eqref{coeffd:eq}, 
\begin{equation*}
|\CB_d(f_R^{(1)})|\lesssim R^{\g-\s} \1 f\1_2.
\end{equation*}
Thus, using \eqref{g2R_le:eq} and the additivity 
\[
D_\a(f) = D_\a(f_R^{(2)}) + D_\a(f_R^{(1)}),\  
\CB_d(f) = \CB_d(f_R^{(2)}) + \CB_d(f_R^{(1)}),
\] 
we get the bound
\begin{equation*}
\limsup\limits_{\a\to\infty }
\bigl|\a^{1-d} 
D_\a(f) - \CB_d(f)\bigr| 
\lesssim \1 f\1_2 R^{\g-\s}.
\end{equation*}
Since $R$ is arbitrary, by taking $R\to 0$, we obtain 
\eqref{main:eq} for the function $f$.
\qed
 
\section{Proof of Theorems \ref{homo:thm}, \ref{homolog:thm}}

Without loss of generality assume that $\|a_\l\|_{\plainL\infty}\le 1$.
We use the notation $f_\l(t) = \l^{-\g} f(\l t)$, $t\in\R$. 

\subsection{Proof of Theorem \ref{homo:thm}}  
Rewrite:
\begin{align*}
\l^{-\g} 
D_\a(\l a_\l, \L; f) = D_\a(a_\l, \L; f_\l),\ \ . 
\end{align*}
Represent the right-hand side as
\begin{align}\label{homoproof:eq}
D_\a(a_\l, \L; f_0) + D_\a(a_\l, \L; g_\l),\ g_\l= f_\l - f_0.
\end{align}
Since $|a_\l|\le 1$, we can replace the function $g_\l$ by $g_\l \z$, where 
$\z\in \plainC\infty_0(\R)$ is a function such that 
$\z(t) = 1$ for $|t|\le 1$, and $\z(t) = 0$ for $|t|\ge 2$. 

By \eqref{crelle2:eq}, the second term satisfies the bound 
\begin{align*}
\|D_\a(a_\l, \L; g_\l)\|_1\lesssim \1 g_\l \z\1_2 \a^{d-1}. 
\end{align*}
Notice that $\1 g_\l\z\1_2 = \1 (f-f_0)\z^{(\l)}\1_2$,\ $\z^{(\l)}(t) = \z(\l^{-1} t)$. 
It is straightforward that the condition 
\eqref{almosthom:eq} implies that $\1 (f-f_0)\z^{(\l)}\1_2\to 0$ as 
$\l\to 0$.
Therefore
\begin{align}\label{glam1:eq}
\a^{1-d} D_\a(a_\l, \L; g_\l)\to 0,\ \a\to\infty, \l\to 0.
\end{align}
By Theorem \ref{main:thm}, the first term in 
\eqref{homoproof:eq} satisfies 
\begin{align*}
\lim\limits_{\a\to\infty}\a^{1-d} D_\a(a_\l, \L; f_0)
= \CB_d(a_\l, \p\L; f_0),
\end{align*}
uniformly in $\l>0$. By Corollary \ref{cbconvf:cor}, the right-hand side converges to 
$\CB_d(a_0, \p\L; f_0)$ as $\l\to 0$. 
Together with \eqref{glam1:eq} this completes 
 the proof. 
\qed
 
\subsection{Proof of Theorem \ref{homolog:thm}} 
Let $f_\l(t) = \l^{-1}f(\l t)$. Similarly to 
the proof of Theorem \ref{homo:thm}, we can rewrite:
\begin{align*}
\l^{-1} 
D_\a(\l a_\l, \L; f) = D_\a(a_\l, \L; f_\l),\ \ . 
\end{align*}
Represent the right-hand side as
\begin{align}\label{homoprooflog:eq}
D_\a(a_\l, \L; h_\l) + D_\a(a_\l, \L; g_\l),\ g_\l= f_\l - h_\l.
\end{align}
Since $|a_\l|\le 1$, we can replace the function $g_\l$ by $g_\l \z$, 
as in the previous proof. 
By \eqref{almosthomlog:eq} $g_\l \z$ satisfies Condition
\ref{f:cond} with $\g = 1$, and hence, 
by \eqref{crelle2:eq}, the second term 
in \eqref{homoprooflog:eq} 
satisfies the bound 
\begin{align*}
\|D_\a(a_\l, \L; g_\l)\|_1\lesssim \1 g_\l \z\1_2 \a^{d-1}. 
\end{align*}
As in the previous proof, 
$\1 g_\l\z\1_2 = \1 (f-h)\z^{(\l)}\1_2$,\ $\z^{(\l)}(t) = \z(\l^{-1} t)$, and  
the condition 
\eqref{almosthomlog:eq} implies the convergence $\1 (f-h)\z^{(\l)}\1_2\to 0$ as 
$\l\to 0$. 
Therefore
\begin{align}\label{glam:eq}
\a^{1-d} D_\a(a_\l, \L; g_\l)\to 0,\ \a\to\infty, \l\to 0.
\end{align}
Since $h_\l(t) = - t\log\l + h(t)$, by Remark \ref{lin:rem}, 
we have that $D_\a(a_\l, \L; h_\l) = D_\a(a_\l, \L; h)$. 
The function $h$ satisfies Condition \ref{f:cond} with arbitrary $\g <1$. 
Thus, by Theorem \ref{main:thm}, the first term in 
\eqref{homoprooflog:eq} satisfies 
\begin{align*}
\lim\limits_{\a\to\infty}\a^{1-d} D_\a(a_\l, \L; h)
= \CB_d(a_\l, \p\L; h),
\end{align*}
uniformly in $\l>0$. By Corollary \ref{cbconvf:cor}, the right-hand side converges to 
$\CB_d(a_0, \p\L; h)$ as $\l\to 0$. 
Together with \eqref{glam:eq}, this 
completes the proof. 
\qed 
 

\begin{thebibliography}{10}
\providecommand{\url}[1]{\texttt{#1}}
\providecommand{\urlprefix}{URL }
\providecommand{\eprint}[2][]{\url{#2}}

\bibitem{BuBu}
A.~Budylin and V.~Buslaev, \emph{On the Asymptotic Behaviour of the Spectral
  Characteristics of an Integral Operator with a Difference Kernel on Expanding
  Domains}. Differential equations, Spectral theory, Wave propagation (Russian)
  \textbf{13}: 16--60, 1991.

\bibitem{GioevKlich2006}
D.~Gioev and I.~Klich, \emph{Entanglement Entropy of Fermions in Any Dimension
  and the Widom Conjecture}. Phys. Rev. Lett. \textbf{\textbf{96}}: 100503,
  2006.

\bibitem{Helling2009}
R.~Helling, H.~Leschke, and W.~Spitzer, \emph{{A Special Case of a Conjecture
  by Widom with Implications to Fermionic Entanglement Entropy}}. Int. Math.
  Res. Not. \textbf{\textbf{2011}}: 1451--1482, 2011.

\bibitem{Hoermander1993}
L.~H\"ormander, \emph{The Analysis of Linear Partial Differential Operators I}.
  Springer-Verlag, Berlin-New York, 1993.

\bibitem{LeschkeSobolevSpitzer2014}
H.~Leschke, A.~V. Sobolev, and W.~Spitzer, \emph{Scaling of R\'enyi
  Entanglement Entropies of the Free Fermi-Gas Ground State: A Rigorous Proof}.
  Phys. Rev. Lett. \textbf{\textbf{112}}: 160403, 2014.

\bibitem{LeschkeSobolevSpitzer2016}
H.~Leschke, A.~V. Sobolev, and W.~Spitzer, \emph{Large-Scale Behaviour of Local
  and Entanglement Entropy of the Free Fermi Gas at Any Temperature}. Journal
  of Physics A: Mathematical and Theoretical \textbf{\textbf{49}(30)}: 30LT04,
  2016.

\bibitem{Leschke2016}
H.~{Leschke}, A.~V. {Sobolev}, and W.~{Spitzer}, \emph{{Trace formulas for
  Wiener--Hopf operators with applications to entropies of free fermionic
  equilibrium states}}. J. Funct. Anal. \textbf{273}: 1049--1094, 2017. \eprint{1605.04429}.

\bibitem{Peller1989}
V.~Peller, \emph{When is a function of a Toeplitz operator close to a Toeplitz
  operator?} \emph{Toeplitz operators and spectral function theory},
  \emph{Oper. Theory Adv. Appl.}, vol.~42, 59--85, Birkh\"auser, Basel, 1989.

\bibitem{Roccaforte1984}
R.~Roccaforte, \emph{symptotic expansions of traces for certain convolution
  operators}. Trans. Amer. Math. Soc. \textbf{285(2)}: 581�--602, 1984.

\bibitem{Sobolev2013}
A.~V. Sobolev, \emph{Pseudo-Differential Operators with Discontinuous Symbols:
  {W}idom's Conjecture}. Mem. Amer. Math. Soc. \textbf{\textbf{222}(1043)}:
  vi+104, 2013.

\bibitem{Sobolev2014}
A.~V. Sobolev, \emph{On the {S}chatten-von {N}eumann Properties of Some
  Pseudo-Differential Operators}. J. Funct. Anal. \textbf{\textbf{266}(9)}:
  5886--5911, 2014.

\bibitem{Sobolev2015}
A.~V. Sobolev, \emph{Wiener-Hopf operators in higher dimensions: the Widom
  conjecture for piece-wise smooth domains}. Integral Equations and Operator
  Theory \textbf{81(3)}: 435--449, 2015.

\bibitem{Sobolev2016b}
A.~V. Sobolev, \emph{On the coefficient in trace formulae for Wiener-Hopf
  operators}. Journal of Spectral Theory \textbf{6(4)}: 1021--1045, 2016.

\bibitem{Sobolev2016}
A.~V. Sobolev, \emph{Functions of Self-Adjoint Operators in Ideals of Compact
  Operators}. Journal of LMS \textbf{95(1)}: 157--176, 2017.

\bibitem{Sobolev2016a}
A.~V. {Sobolev}, \emph{{Quasi-Classical Asymptotics for Functions of
  Wiener-Hopf Operators: Smooth vs Non-Smooth Symbols}}. Geom. Funct. Anal.
  \textbf{27(3)}: 676--725, 2017. \eprint{1609.02068}.

\bibitem{Szego1915}
G.~Szeg{\H o}, \emph{Ein Grenzwertsatz {\"u}ber die Toeplitzschen Determinanten
  einer reellen positiven Funktion}. Math. Ann \textbf{76}: 409--503, 1915.

\bibitem{Szego1952}
G.~Szeg{\H o}, \emph{On Certain {H}ermitian Forms Associated with the {F}ourier
  Series of a Positive Function}. Comm. S\'em. Math. Univ. Lund, Tome
  Suppl\'ementaire 228--238, 1952.

\bibitem{Widom1980}
H.~Widom, \emph{Szeg{\H o}'s limit theorem: the higher-dimensional matrix
  case}. J. Funct. Anal. \textbf{39(2)}: 182--198, 1980.

\bibitem{Widom1982}
H.~Widom, \emph{A trace formula for Wiener-Hopf operators}. Journal of Operator
  Theory \textbf{8(2)}: 279--298, 1982.

\bibitem{Widom1985}
H.~Widom, \emph{Asymptotic Expansions for Pseudodifferential Operators on
  Bounded Domains}, \emph{Lecture Notes in Mathematics}, vol. 1152.
  Springer-Verlag, New York-Berlin, 1985.

\end{thebibliography}

\end{document}